\documentclass[smallextended,envcountsect,envcountsame]{svjour3} 
\usepackage{graphicx}
\usepackage{mathrsfs}
\usepackage{fix-cm}
\usepackage{amsmath}
\usepackage{amssymb}
\usepackage{latexsym}
\usepackage{dsfont}
\usepackage{xcolor}
\usepackage{cite}
\usepackage{dsfont}
\usepackage{enumitem}

\def\disp{\displaystyle}

\def\Limsup{\mathop{{\rm Lim}\,{\rm sup}}}

\def\tto{\;{\lower 1pt \hbox{$\rightarrow$}}\kern -10pt
\hbox{\raise 2pt \hbox{$\rightarrow$}}\;}

\def\Hat{\widehat}

\def\Bar{\overline}
\def\ra{\rangle}
\def\la{\langle}
\def\ve{\varepsilon}
\def\B{\mathbb{B}}

\def\R{\mathbb{R}}
\def\N{\mathbb{N}}

\def\ox{\bar{x}}
\def\oy{\bar{y}}
\def\ooy{\bar{\y}}
\def\oz{\bar{z}}

\def\ow{\bar{w}}

\def\gph{\operatorname{gph}}
\def\epi{\mbox{\rm epi}\,}

\def\dom{\mbox{\rm dom}\,}
\def\ker{\mbox{\rm ker}\,}
\def\proj{\mbox{\rm proj}\,}
\def\lip{\mbox{\rm lip}\,}
\def\reg{\mbox{\rm reg}\,}

\def\dist{\mbox{\rm dist}\,}

\def\substack#1#2{{\scriptstyle{#1}\atop\scriptstyle{#2}}}

\def\ph{\varphi}
\def\emp{\emptyset}
\def\st{\stackrel}
\def\oR{\Bar{\R}}
\def\N{\mathbb{N}}

\def\gg{\gamma}
\def\tt{\tilde t}
\def\dd{\delta}

\def\N{{\rm I\!N}}

\newcounter{count}

\newcommand{\Intf}[1]{\mathrm{E}_{#1}}
\newcommand{\Intfset}[1]{\mathrm{E}_{#1}}

\DeclareMathOperator*{\esssup}{ess\,sup}

\newcommand{\T}{T}
\newcommand{\Rex}{\overline{\mathbb{R}}}
\let\epsilon\varepsilon
\DeclareMathAlphabet{\mathpzc}{OT1}{pzc}{m}{it}

\def\v{\mathpzc{v}} 
\def\w{\mathpzc{w}}
\def\x{\mathpzc{x}}
\def\y{\mathpzc{y}}
\def\z{\mathpzc{z}}
\def\X{{\R^n} }
\def\Y{\R^{m} }
\def\Z{\R^{q}} 
\def\Leb{\textnormal{L} }
 
\def\dmu{\mu(dt)}
\def\disp{\displaystyle}

\begin{document}

\title{Sensitivity Analysis of Stochastic Constraint and Variational Systems via Generalized Differentiation
\thanks{Research of the first author was partly supported by the USA National Science Foundation under grant DMS-1808978, by the Australian Research Council under Discovery Project DP-190100555, and by the Project 111 of China under grant D21024. Research of the second author was partly supported by the Chilean grants: Fondecyt Regular 1190110 and Fondecyt Regular 1200283.}}
\author{Boris S. Mordukhovich \and \mbox{Pedro P\'erez-Aros}}

\institute{Boris S. Mordukhovich \at Department of Mathematics, Wayne State University, Detroit, Michigan 48202, USA\\ \email{aa1086@wayne.edu} \\
\and Pedro P\'erez-Aros \at Instituto de Ciencias de la Ingenier\'ia, Universidad de O'Higgins, Rancagua, Chile\\
\email{pedro.perez@uoh.com}}

\date{Received: date / Accepted: date}

\maketitle

\begin{abstract}
This paper conducts sensitivity analysis of random constraint and variational systems related to stochastic optimization and variational inequalities. We establish efficient conditions for well-posedness, in the sense of robust Lipschitzian stability and/or metric regularity, of such systems by employing and developing  coderivative characterizations of well-posedness properties for random multifunctions and efficiently evaluating coderivatives of special classes of random integral set-valued mappings that naturally emerge in stochastic programming and stochastic variational inequalities.\vspace*{-0.05in}

\vspace*{-0.05in}

\vspace*{-0.05in}
\end{abstract}\vspace*{-0.15in}

\section{Introduction}\label{intro}\vspace*{-0.1in}\setcounter{equation}{0}
Sensitivity analysis has drawn a strong attention of many researchers and users in the areas of optimization, variational analysis, and related disciplines who are interested in understanding the behavior of feasible and optimal solutions
under perturbations of the initial data. Such perturbations should be taken into account due to ``always present" errors in the given data. The literature on sensitivity analysis for various classes of optimization-related problems in deterministic frameworks is enormous; see, e.g., \cite{bs,m06} and the bibliographies therein.  To the best of our knowledge, much less has been done in this direction for stochastic problems; we refer the reader to \cite{sdr} for a very recent account.

Variational analysis offers natural approaches to the study of sensitivity of parametric sets of feasible and optimal solutions to optimization-related and equilibrium problems with respect to parameter perturbations. Among such approaches, we emphasize those based on {\em generalized differentiation} of usually {\em set-valued} parameter-dependent solution maps. Efficient results of this type for {\em robust Lipschitzian stability} of solution maps associated with {\em deterministic} constraint and variational systems were developed in \cite{m06,m18} based on the {\em coderivative} concept for set-valued mappings (multifunctions) introduced in \cite{m80}. This approach is based on the {\em complete characterization} of robust Lipschitzian stability of general closed-graph multifunctions (and equivalent properties of {\em metric regularity} and {\em linear openness} of inverse mappings) obtained in \cite{m93} and known as the {\em Mordukhovich criterion} \cite{rw}. Due to the well-developed calculus rules and computations of coderivatives, this criterion has been broadly applied to the sensitivity analysis of various deterministic constraint and variational systems; see \cite{m06,m18} and the references therein.

The situation is much more involved for {\em random multifunctions} in all the aspects: coderivative calculus, a variety of Lipschitzian properties, and their coderivative characterizations. Various results in this directions for general classes of {\em expected-integral} mappings have been recently obtained in our papers \cite{mp2,mp3}. For the reader's convenience, the major results obtained therein, which are needed in what follows, will be briefly reviewed in the next section. 

The {\em main goal} of this paper is to elaborate and further develop the aforementioned results for general random multifunctions in order to apply them to the sensitivity analysis of {\em structured random constraint} and {\em variational systems} that appear in various models of {\em stochastic optimization} and {\em stochastic variational inequalities} with applications to {\em well-posedness} of such systems in the sense of their robust Lipschitzian stability and/or metric regularity. 

The rest of the paper is organized as follows. In Section~\ref{SECTION:2} we {\em overview} the major constructions and results of variational analysis, generalized differentiation, random measurable multifunctions and their integration, which are largely used in the formulation and derivation of the main results of the paper. 

Section~\ref{Section:Lipschitz-Likeproperty} is devoted to various {\em Lipschitzian} properties of {\em random normal integrands} and the corresponding {\em expected integral multifunctions}. Using coderivatives, we establish efficient conditions ensuring these properties, which are significant for the subsequent material.

In Section~\ref{Sensitivity_analysis_parametric_constraint} we conduct a {\em coderivative-based sensitivity analysis} for a general class of {\em stochastic constraint systems} as well as for their various specifications. Our approach is based on coderivative evaluations (upper estimates) for such parametric systems in terms of the initial data with further applications of coderivative conditions ensuring their {\em well-posedness properties} such as {\em Lipschitzian stability} and {\em metric regularity}.

Section~\ref{Sensitivity_variational} concerns {\em stochastic variational systems}, which are described by {\em stochastic generalized equations} involving expected integral multifunctions and encompass, in particular, {\em stochastic variational inequalities}. We first evaluate coderivatives of solution maps to such systems in terms of their initial data and then use these calculations of derive efficient conditions for their well-posedness properties based on coderivative characterizations. 

Section~\ref{Sen_constraint_integrable_LipsLike} provides a coderivative-based sensitivity analysis of solution maps to stochastic constraint and variational systems in the case where the random integrand enjoys the {\em integrable Lipschitz-like} property being also single-valued at the reference point. Finally, in Section~\ref{Appli_Stochastic} we apply the obtained results to explicit coderivative evaluations for {\em stationary point maps} in stochastic programming and to the derivation of necessary optimality conditions in {\em stochastic mathematical programs with equilibrium constraints}.
\vspace*{-0.25in}
  
\section{Preliminaries from variational analysis and set-valued integration}\label{SECTION:2}\vspace*{-0.1in}\setcounter{equation}{0}

In this section we present some background from variational and set-valued analysis needed in what follows; see the books \cite{cv,m06,rw} for more details and references. Throughout the paper we use the standard notation, which can be found in these books. Recall that $\N:=\{1,2,\ldots\}$ and that the extended real line is denoted by $\Rex:=[-\infty,+\infty]$ with the usual convention of $(+\infty)+(-\infty): =+\infty$ and $0\cdot (\pm \infty):=0$. For $x\in \X$ and $r>0$, the closed ball of radius $r$ centered at $x$ is denoted by $\mathbb{B}_r(x)$, while the unit closed ball is written as $\mathbb{B}$. The symbol $^\top$ stands for the vector and matrix transposition.\vspace*{-0.25in}

\subsection{\bf Variational analysis and generalized differentiation}\label{subsec:va}\vspace*{-0.1in}

Here we recall some major notions of generalized differentiation for sets, set-valued mappings/multifunctions, and extended-real-valued functions that are broadly employed in the paper.

Given $C\subseteq\X$, the (Fr\'echet) {\em regular normal cone} to $x\in C$ is defined by
\begin{equation}\label{rnc}
\Hat{N}(x;C):=\Big\{ x^*\in \X\;\Big|\;\limsup\limits_{ u \overset{C}{\to} x  }  \Big\langle x^\ast ,\frac{u - x}{\| u - x  \| }\Big\rangle\leq 0\Big\},
\end{equation}
where the symbol `$y\overset{C}{\to} x$' means that $y \to x$ with $y \in C$. We put $\Hat N(x;C):=\emp$ if $x\notin C$. The (Mordukhovich) {\em limiting normal cone} to $C$ at $x\in\R^n$ is 
\begin{equation}\label{lnc}
N(x;C):=\Limsup\limits_{u \to x}\Hat{N}(u;C),
\end{equation}  
where for any multifunction $F: \X \tto \Y$ the symbol `$\Limsup$' stands for the (Painlev\'e-Kuratowski) {\em outer limit} of $F$ at $x$ defined by
\begin{equation}\label{pk}
\Limsup_{u\to x}F(u):=\big\{y\in\Y\big|\;\exists\,u_k\to x,\,y_k\to y,\;y_k\in F(u_k)\;\mbox{as}\;k\in\N\big\}.
\end{equation}
The set $C$ is called {\em normally regular} at $x\in C$ if $\Hat N(x;C)=N(x;C)$.

Given a multifunction $F: \X \tto \Y$, denote its {\em graph} and {\em kernel} by
\begin{equation*}
\gph F:=\big\{ (v,w) \in \X\times \Y\;\big|\; w\in F(x)\big\}\;\mbox{ and }\;\ker F:=\big\{ x \in \X\;\big|\;0\in F(x)\big\}.
\end{equation*} 
The {\em regular coderivative} and {\em limiting coderivative} of $F$ at $(x,y)\in\gph F$ are defined  for all $y^*\in\Y$ via the corresponding normal cone \eqref{rnc} and \eqref{lnc} to the graph of $F$ by, respectively,
\begin{equation}\label{rcod}
\Hat{D}^\ast F(x,y) (y^\ast):=\big\{  x^\ast \in \X\;\big|\;(x^\ast,-y^\ast) \in  \Hat{N}\big( (x,y); \gph F\big)\big\},
\end{equation}
\begin{equation}\label{lcod}
{D}^\ast F(x,y) (y^\ast):=\big\{x^\ast \in \X\;\big|\;(x^\ast,-y^\ast) \in  N\big( (x,y); \gph F\big)\big\},
\end{equation}
where $y$ is dropped if $F$ is a singleton at $x$, i.e., $F(x)=\{y\}$. From \eqref{lnc} we have the limiting representation
\begin{align}\label{Lim_repr_cod}
{D}^\ast F(x,y) (y^\ast)= \Limsup\limits_{ 	\substack{   (v,w) \overset{  \gph F }{\longrightarrow } (x,y)}{ w^\ast \to y^\ast }	}  \Hat{D}^\ast F(v,w) (w^\ast),\quad y^*\in\Y.
\end{align}

Given now an extended-real-valued function $f:\X\to\Rex$, assume in what follows that it is {\em proper}, i.e., $f(x)>-\infty$ for all $x\in \X$ and its   domain $\dom f := \{ x\in X\;\big|\;f(x) < +\infty \}$ is nonempty. Considering the epigraph $\epi f:=\{ (x,\alpha) \in \X\times \R \;|\;f(x)\leq \alpha\}$ of $f$, we define its (limiting) {\em subdifferential} at $x\in\dom f$ geometrically by 
\begin{equation}\label{sub}
\partial f(x) := \big\{ x^\ast\in \mathbb{R}^n\;\big|\;(x^\ast ,-1) \in N\big((x,f(x));\epi f\big)\big\},
\end{equation}
while referring the reader to \cite{m06,m18,rw} for various limiting analytic representations of \eqref{sub} as well as comprehensive {\em calculus rules} for this subdifferential and the associated limiting constructions \eqref{lnc} and \eqref{lcod}. Observe that if a mapping $F\colon\X\to\Y$ is single-valued and locally Lipschitzian around $x$, then we have the coderivative {\em scalarization formula} held for all $y^*\in\Y$:
\begin{equation}\label{scal}
{D}^\ast F(x) (y^\ast) =\partial \langle y^\ast, F \rangle (x)\text{ with } \langle y^\ast, F \rangle (\cdot):= \langle y^\ast, F (\cdot)\rangle.
\end{equation}

Along with the (first-order) subdifferential \eqref{sub}, consider the {\em second-order subdifferential} (or {\em generalized Hessian}) of $f: \X \to \Rex$ at $x\in \dom f$ relative to $x^\ast\in\partial f(x)$ introduced in \cite{m92} by 
\begin{equation}\label{2nd}
\partial^2f(x,x^\ast) (v^\ast) = \left( D^\ast\partial f\right) (x,x^\ast)(v^\ast),\quad v^\ast\in \X.
\end{equation}
Dealing with extended-real-valued functions of two variables $f(x,z):\X\times \Z \to \Rex$, we use the {\em partial second-order subdifferential} of $f$ with respect to $x$ at $(\ox,\oz)$ relative to $\oy\in\partial_x f(\ox,\oz)$ defined by
\begin{equation}\label{2par}
\partial_x^2 f(\ox,\ow,\oy)(u^\ast):= \left( D^\ast \partial_z f\right)(\ox,\ow,\oy)(u^\ast),
\end{equation}
where $\partial_x f(x,z):=\partial f_z(x)$ with $f_z:=f(\cdot,z)$; see \cite{mr}. Note that both second-order constructions \eqref{2nd} and \eqref{2par} enjoy well-developed calculus rules and are efficiently computed for broad classes of functions encountered in variational analysis and optimization; see \cite{m06,m18,mr} and the references therein.\vspace*{0.03in}

Next we recall the two interrelated {\em well-posedness} properties of multifunctions, which play a fundamental role in many aspects of variational analysis, optimization, and applications. They both are studied and largely used in the paper. We say that $F\colon\X\tto\Y$ is {\em Lipschitz-like} around $(x,y)\in\gph F$ (or satisfies the {\em Aubin pseudo-Lipschitz property} around this point) if 
there exist a constant $\ell\ge 0$ and neighborhoods $U$ of $x$ and $V$ of $y$ such that 
\begin{equation}\label{lip}
F(u) \cap V \subseteq F(v) + \ell \| u-v\| \mathbb{B}\;\text{  for all }\;u,v \in U.
\end{equation}
A complete characterization of the Lipschitz-like property of closed-graph multifunctions $F$ around $(x,y)$ and the precise calculations of the exact bound $\lip F(x,y)$, i.e., the infimum of moduli $\ell$ over neighborhood $U$ and $V$ in \eqref{lip}, are given by the {\em Mordukhovich coderivative criterion}  
\begin{equation}\label{cod-cr}
D^*F(x,y)(0)=\{0\}\;\mbox{ with }\;\lip F(x,y)=\|D^*F(x,y)\|
\end{equation}
in terms of the limiting coderivative \eqref{lcod} and its norm as a positive homogeneous mapping of $y^*$. Another well-posedness property used in this paper is the {\em metric regularity} of $F$ around $(x,y)\in\gph F$ defined as
\begin{equation}\label{metreg}
\dist\big(u; F^{-1}(v)\big)\le\kappa\,\dist\big(v;F(u)\big)\;\text{ for all }\;u\in U\;\text{ and }\;v\in V
\end{equation}
for some constant $\kappa\ge 0$ and neighborhoods $U$ of $x$ and $V$ of $y$ with $\reg F(x,y)$ standing for the exact bound (supremum) of $\kappa$ over $U$ and $V$ in \eqref{metreg}. As well known in variational analysis, the Lipschitz-like property of $F$ around $(x,y)$ is {\em equivalent} to the metric regularity of the inverse mapping $F^{-1}$ around $(y,x)$ with the exact bound relationship $\lip F(x,y)\cdot\reg F^{-1}(y,x)=1$. Thus the coderivative criterion \eqref{cod-cr} yields the {\em metric regularity characterization}:
\begin{equation}\label{mr}
\ker D^*F(x,y)=\{0\}\;\mbox{ with }\;\reg F(x,y)=\|D^*F(x,y)\|^{-1}.
\end{equation}
Since the coderivatives in \eqref{cod-cr} and \eqref{mr} are {\em robust} and possess {\em full calculus}, both these characterizations are broadly used in variational analysis, optimization, and their applications; see \cite{m06,m18,rw} for more details and references.\vspace*{-0.2in}

\subsection{\bf Measurable multifunctions and expected-integral mappings}\label{sec:expected}\vspace*{-0.1in}

In the second part of this section, we review some required classical notions and results on measurable multifunctions and their selections, as well as more recent ones dealing with expected-integral functionals and set-valued mappings. 
Throughout the paper, $({T},\mathcal{A},\mu)$ is a complete finite measure space. To avoid confusion, we use special font (e.g., $\v, \w, \x,\y, \z$, etc.) to designate functions defined on $T$. For any $p\in[1,+\infty]$, denote by $\Leb^p({\T},\R^n)$ the sets of all (equivalence classes by the relation `equal almost everywhere') measurable functions $\x$ such that $\| \x(\cdot)\|^p$ is integrable for $p\in[1,+\infty)$ and measurable essentially bounded functions for $p=\infty$. The norm in $\Leb^p(T, \mathbb{R}^n)$ is denoted by $\|\cdot\|_p$. We identify points in $\X$ with constant functions in $\Leb^p(T,\X)$ and so have
\begin{align*}
\| x- \x \|_p: &= \left( \int_T \| x - \x(t)\|^p \dmu \right)^{1/p}\text{as  }\; p \in [1,+\infty),\\
\| x- \x \|_\infty: & = \esssup\limits_{t\in T } \| x- \x(t)\|\;\mbox{ for $x\in \X$ and $\x \in \Leb^p(T,\X)$}.
\end{align*} 

Recall that a set-valued mapping $M:T\tto\X$ is 
\emph{measurable} if for every open set $U \subseteq \mathbb{R}^n$ the set 
$M^{-1}(U):=\{ t \in T\;|\; M(t)\cap U \neq \emptyset\}$ is measurable, i.e., $M^{-1}(U)\in\mathcal{A}$. The mapping $M$ is \emph{graph measurable} if $\gph M \in \mathcal{A}\otimes\mathcal{B}(\mathbb{R}^n)$, where $\mathcal{B}(\mathbb{R}^n)$ is the Borel $\sigma$-algebra, i.e., the $\sigma$-algebra generated by all open sets of $\mathbb{R}^n$. Since the space $(T,\mathcal{A},\mu)$ is complete, any multifunction $M$ with closed values is measurable if and only if $\gph M \in \mathcal{A}\otimes \mathcal{B}(\X)$.

For a multifunction $M:{\T}\tto \R^n$ (not necessarily measurable) and a measurable set $A \in \mathcal{A}$, define the \emph{the Aumann integral} of $M$ over $A$ by  
$$
\int\limits_A M(t) \dmu  :=\Bigg\{ \int\limits_A \x^*(t) \dmu\;\Bigg|\;x^* \in {\Leb}^1({\T},\X) \textnormal{ and } \x^*(t) \in M(t) \text{ a.e.}\Bigg\}. 
$$
An extended-real-valued function  $f:T\times \mathbb{R}^n \to \Rex$ is called  a \emph{normal integrand} if the multifunction $t\mapsto\epi f_t$ is measurable with closed values. By \cite[Corollary~14.34]{rw}, this amounts to saying that $f$ is $\mathcal{A}\otimes\mathcal{B}(\mathbb{R}^n)$-measurable and that for every $t\in {T}$ the function $f_t:=f(t,\cdot)$ is lower semicontinuous function (l.s.c.). In addition, we say that $f$ is {\em proper} if $f_t$ is proper for all $t\in T$. If $f_t$ is convex for all $t \in T$, then $f$ is known as a {\em convex normal integrand}; see \cite{cv,rw}.

We say \cite{mp2} that $\Phi: T \times  \X \tto  \Y$  is a {\em set-valued normal integrand }, or a {\em random multifunction}, if for all $t \in T$ the mapping $\Phi_t:=\Phi(t,\cdot)$ is of closed graph, and the graph of $\Phi $ belongs to $\mathcal{A}\otimes\mathcal{B}(\X\times \Y)$. When $\Phi$ is single-valued, it is called a {\em vector-valued normal integrand}. Since the measure space is complete, we have that the definition of a set-valued normal integrand is equivalent to requiring that $t\mapsto\gph \Phi_t$ is a measurable multifunction with closed values. It is natural to say that $\Phi: T \times  \X \tto  \Y$ is {\em locally single-valued} around $\ox \in \X$ if there exists $\eta >0 $ and $\Hat T \in \mathcal{A}$ with $\mu(T\backslash \Hat T)$ such that
\begin{equation*}
\Phi_t(x) \text{ is single-valued for all } x\in \mathbb{B}_\eta (\ox) \text{ and all }  t\in  \Hat T.
\end{equation*}
The term that $\Phi$ is {\em continuously differentiable} around $\ox$ is defined similarly.\vspace*{0.03in}

We next result is taken from \cite[Proposition~3.2]{mp3}.\vspace*{-0.05in}
\begin{lemma}\label{lemma_measurability_reg_sub}
Let $f: T\times \X \to \Rex$ be a  normal integrand, and  let $\Phi: T \times  \X \tto  \Y$   be a set-valued normal integrand with respect to the measure space $({T},\mathcal{A},\mu)$. Then the following multifunctions are graph measurable:
\begin{enumerate}[label=\alph*)]
\item[{\bf(a)}] $t\mapsto\gph  \partial f_t =\big\{ (x,x^\ast) \in \mathbb{R}^{2n}\;\big|\;x^\ast \in \partial f_t (x)\big\}$, 
\item[{\bf(b)}]  $t\mapsto\gph \Hat D^\ast \Phi_t=\big\{ (x,y, x^\ast, y^\ast) \in \mathbb{R}^{2(n+m)}\;\big|\;x^\ast \in  \Hat D^\ast \Phi_t (x,y)(y^\ast)\big\} $, 
\item[{\bf(c)}] $t\mapsto\gph D^\ast  \Phi_t = \big\{ (x,y, x^\ast, y^\ast) \in \mathbb{R}^{2(n+m)}\;\big|\;x^\ast \in D^\ast \Phi_t (x,y)(y^\ast)\big\} $.
\end{enumerate}
\end{lemma}

Finally in this section, we recall the notions of expected-integral functional and expected-integral multifunctions studied in \cite{mp2,mp3} and in what follows. Given a normal integrand $f:T\times \X \to \Rex$, define the {\em expected-integral functional} associated with $f$ by
\begin{align}\label{def:set-valued:Exp0}
\Intfset{f}(x):=\int_{ T  } f_t(x)\dmu=\int_T\Big[\max\big\{  f_t(x),0\big\}+ \min\big\{  f_t(x),0\big\}\Big]\dmu.
\end{align}
For a set-valued normal integrand $\Phi : T\times \X \to \Y$, the {\em expected-integral multifunction} is defined by
\begin{align}\label{def:set-valued:Exp}
\Intfset{\Phi}(x):=\int_{ T  } \Phi_t(x)\dmu.
\end{align}
Observe that for $\Phi(t,x):=\{ \alpha \in \mathbb{R}\;|\;f_t(x)\leq \alpha		\}$ we have $\epi \Intf{f} = \gph \Intfset{\Phi}$. 

Our {\em standing assumptions} in the subsequent study of expected-integral multifunctions of type \eqref{def:set-valued:Exp} around a reference point $\ox\in\dom \Intfset{\Phi}$ are formulated as follows: there exists a number $\rho>0$ such that   
\begin{align}\label{convex_cond_set} 
\begin{aligned}
\Phi_t(x) &\text{ is convex for all } x\in \mathbb{B}_\rho(\ox) \text{ and all } t\in T_{na},\\
\Phi_t(x) &\subseteq \kappa(t) \mathbb{B} \text{  for all } x \in \mathbb{B}_\rho(\ox) \text{ and all } t\in T,
\end{aligned}
\end{align}
where $T_{na}$ is the nonatomic part of the measure $\mu$. We associate with the integrand $\Phi$ the set-valued mapping
$\mathcal{S}_\Phi$ defined by
\begin{align}\label{mapping:SPHI}
\mathcal{S}_{\Phi}(x,y):=\left\{\y\in\Leb^1(T,\Y)\Big|\int_T\y(t)\dmu=y\text{ and }\y(t)\in \Phi_t(x)\text{ a.e.}\right\}.
\end{align}\vspace*{-0.35in}

\section{Lipschitzian properties for random multifunctions}\label{Section:Lipschitz-Likeproperty}\vspace*{-0.05in}\setcounter{equation}{0}

This section is devoted to the study of various {\em Lipschitzian} properties 
of {\em random} set-valued mappings that are largely used in our sensitivity analysis in what follows. These properties were introduced in \cite{mp3} with establishing relationships between them and their applications to coderivative calculus of random multifunctions. First we recall the formulations.\vspace*{-0.05in}

\begin{definition}\label{int-lipl} Let $\Phi: T\times \X \tto \Y$ be  a set-valued normal integrand. 
\begin{enumerate}
\item[\bf(i)] We say that $\Phi$ is {\em integrably locally Lipschitzian} at $\ox\in\dom\Phi$ if there exists $\eta >0$, $\ell \in \Leb^1(T,\R)$, and $\Hat T \in\mathcal{A}$ with $\mu(T \backslash \Hat T)$ such that  
\begin{align}\label{eq_definition_Int_loc}
\Phi_t(x)\subseteq \Phi_t(x') + \ell(t)\| x - x'\|\mathbb{B}\text{ for all } t\in \Hat T \text{ and } x,x'\in \mathbb{B}_\eta(\ox).
\end{align}

\item[\bf(ii)] Pick $(\ox,\oy)\in\gph\Intfset{\Phi}$ with $\gph\Intfset{\Phi}$ defined in \eqref{def:set-valued:Exp} and take $\ooy\in\mathcal{S}_{\Phi}(\ox,\oy)$ for ${\cal S}_\Phi$ from \eqref{mapping:SPHI}. We say that $\Phi$ is {\em integrably quasi-Lipschitzian} around $(\ox,\ooy)$ if there exist $\eta>0$ and $\ell\in\Leb^1(T,\R_+)$ such that 
\begin{equation}\label{Int_Lips_like_inq}
\sup\big\{\|x^\ast\|\;\big|\;x^\ast\in{D}^\ast\Phi_t\big(\x(t),\y(t)\big)\big(\y^\ast(t)\big)\big\}\le\ell(t)\|\y^\ast(t)\|
\end{equation}for a.e.\ $t\in T$ and all
$\x\in\mathbb{B}_\eta(\ox)$, $\y\in\mathbb{B}_\eta(\ooy)\cap\Phi(\x)$, $\y^\ast\in\Leb^\infty(T,\Y)$ with 
\begin{equation*}
\mathbb{B}_\eta(\ooy)\cap\Phi(\x):=\big\{\y\in\Leb^1(T,\Y)\;\big|\;\y\in\mathbb{B}_\eta(\ooy)\;\text{ and }\;\y(t)\in\Phi_t\big(\x(t)\big)\;\text{ a.e.}\big\}. 
\end{equation*}
	
\item[\bf(iii)] Let $(\ox,\oy)\in\gph\Intfset{\Phi}$, and let $\ooy\in\mathcal{S}_{\Phi}(\ox,\oy)$ for ${\cal S}_\Phi$ taken from
\eqref{mapping:SPHI}. We say that $\Phi$ is {\em integrably Lipschitz-like} around $(\ox,\ooy)$ if there exist positive constants $\ell, \eta,\gg$ and a measurable set $\Hat T\in\mathcal{A}$ with $\mu(T\backslash\Hat T)=0$ such that 
\begin{equation}\label{eq_definition_Int_loc:Lipschitz-like}
\Phi_t(x)\cap\B_{\gg}\big(\y(t)\big)\subset\Phi_t(x')+\ell\|x-x'\|\mathbb{B}
\end{equation} 
for all $t\in\Hat T\;\text{ and }\;x,x'\in\mathbb{B}_{\eta}(\ox)$.
\end{enumerate}
\end{definition}

Note that if $\Phi$ is locally single-valued and Lipschitz continuous on 
$\mathbb{B}_\eta(\ox)$, then \eqref{eq_definition_Int_loc} reduces to the condition
\begin{align}\label{eq:Lipsc}
\| \Phi_t( x) - \Phi_t( x')\| \leq \ell(t)   \| x- x'\| \text{ for   all } t\in \Hat T, \text{ and all } x,x'\in \mathbb{B}_\eta(\ox).
\end{align}
Similarly we can rewrite \eqref{eq_definition_Int_loc:Lipschitz-like} in such a case.\vspace*{0.05in}

Our first result here derives coderivative conditions from the above Lipschitzian properties of random multifunctions.\vspace*{-0.05in}

\begin{theorem}\label{Prop61}
Let $\Phi: T\times \X \tto \Y$ be a set-valued normal integrand, let $\ox\in \dom \Phi $, and  let  $\oy \in \Intfset{\Phi}(\ox)$. Then we have the following assertions:
\begin{enumerate}[label=\alph*)]
\item[\bf(i)]\label{Prop61_case_a}  Let $\ooy \in \mathcal{S}_\Phi(\ox,\oy)$ and assume that  $\Phi$   satisfies the integrable quasi-Lipschitz condition \eqref{Int_Lips_like_inq} around $(\ox,\ooy)$. Then there exists $\ell \in \Leb^1(T,\R)$ such that  for every $\ooy^\ast \in \Leb^\infty(T, \Y)$  we have 
\begin{align}\label{Prop61_eq01}
\sup\left\{ \| x^\ast\|\;\big|\;x^\ast \in {D}^\ast\Phi_t\big(\ox,\ooy(t)\big)(\y^\ast(t)) \right\} \leq \ell(t)\|\y^\ast(t)\| \text{ a.e. }
\end{align}
\item[\bf(ii)] \label{Prop61_case_b}  $\Phi$ satisfies the   integrable locally Lipschitz condition \eqref{eq_definition_Int_loc} at $\ox$, then there exist $\eta >0$, $\ell \in \Leb^1(T,\R)$  and    $\Hat T \in\mathcal{A}$ with $\mu(T \backslash \Hat T)$ such that for all $y^\ast$
\begin{align}\label{Prop61_eq02}
\sup\left\{ \| x^\ast\| : x^\ast \in  {D}^\ast \Phi_t(x,y)(y^\ast) \right\} \leq \ell(t)\|y^\ast\|, \; \forall y \in \Phi_t(x), \, \forall t\in \Hat T.
\end{align}
\end{enumerate}
\end{theorem}
\begin{proof}
Suppose without loss of generality that $\mu(T)=1$. To verify the first assertion of the theorem,  take $\ell \in \Leb^1(T,\R)$ and $\eta >0$ from the definition of the integrable quasi-Lipschitzian 
property \eqref{Int_Lips_like_inq} and pick any $\gamma \in (0,\eta) $. Then Lemma~\ref{lemma_measurability_reg_sub} tells us that the function
$$
t\to \rho(t):=\sup\left\{ \| x^\ast\|\;\big|\; x^\ast \in {D}^\ast \Phi_t(\ox,\ooy(t))(\y^\ast(t)) \right\}
$$ is measurable, and hence the set $\Hat T:= \{ t \in T\;|\;\rho(t)> 0\}$ is measurable as well. Having $ {D}^\ast \Phi_t(\ox,\ooy(t))(\y^\ast(t)) \neq \emptyset$ for all $t\in \Hat T$, define the multifunction $M_\gamma:\Hat  T\tto \R^{2(n+m)}$ by
\begin{align*}
(u,u^\ast, v ,v^\ast) \in M(t) \Longleftrightarrow \begin{array}{c}
u^\ast \in\Hat {D}^\ast \Phi_t(u,v)(v^\ast), u \in \mathbb{B}_\gamma(\ox), v \in \mathbb{B}_{ \gamma }\big(\ooy(t)\big), \\
\rho(t) \leq \| u^\ast \| + \eta, \; \| v^\ast - \ooy^\ast(t)\| \leq \gamma/ (1+ \ell(t)).
\end{array}
\end{align*}
If follows from Lemma~\ref{lemma_measurability_reg_sub} and the measurability of the functions above that $M$ is graph measurable. Furthermore, the definitions of the coderivative and of the functions $\rho$ ensure that $M(t)\ne\emp$ for all $t\in \hat T$. The classical measurable selection theorem (see, e.g., \cite[Theorem~III.22]{cv}) gives us a measurable quadruple $(u(t),u^\ast(t), v(t) ,v^\ast(t)) \in M(t)$ for all $t\in \Hat T$. Employing now the quasi-Lipschitzian condition \eqref{Int_Lips_like_inq}, we 
get
\begin{align*}
\rho(t) &\leq \| u^\ast(t) \| + \eta \leq \ell(t) \| v^\ast(t) \|  + \eta 
\le\ell(t) \| \ooy^\ast(t) \| + \ell(t)\| v^\ast - \ooy^\ast(t)\| + \eta \\
&\leq  \ell(t) \| \ooy^\ast(t) \| +2 \eta\;\mbox{ for all }\;t\in\Hat T.
\end{align*}
Since $\eta>0$ was chosen arbitrarily, we arrive at \eqref{Prop61_eq01} and  justify assertion (i).

To proof the second assertion, pick $\eta >0$, $\ell \in \Leb^1(T,\R)$,  and    $\Hat T \in\mathcal{A}$ with $\mu(T \backslash \Hat T)$ such that  the   integrable locally Lipschitzian condition \eqref{eq_definition_Int_loc} is satisfied around $\ox$. Then fix $x \in \mathbb{B}_{\eta/2}(\ox)$, $t\in \Hat T$, and $y \in \Phi_t(x)$ and $y^\ast\in \Y$. It follows from \cite[Theorem~4.4]{mp3} that 
\begin{align*} 
\sup\left\{ \| v^\ast\|\;\Big|\;v^\ast \in \Hat{D}^\ast \Phi_t(v,w)(w^\ast) \right\} \leq \ell(t)\|w^\ast\|
\end{align*}
for all $v\in\B_{\eta/3}(\ox)$, $w\in\Phi_t(v)$, and $w^*\in\Y$. Using 
the limiting coderivative representation \eqref{Lim_repr_cod}, we arrive at \eqref{Prop61_eq02} and thus complete the proof.
\end{proof}\vspace*{-0.05in}

Before presenting the main result of this section, we formulate the next lemma taken from \cite{mp3}, which provides {\em coderivative Leibniz rules} for expected-integral multifunctions. Recall that the mapping $\mathcal{S}_\Phi$ from \eqref{mapping:SPHI} is {\em inner semicompact} at $(\ox,\oy)$ if for every sequence $(x_k,y_k)\to(\ox,\oy)$ there exists a sequence $\y_k\in\mathcal{S}_\Phi(x_k,y_k)$ containing an $\Leb^1({T},\Y)$-norm convergent subsequence.\vspace*{-0.05in}

\begin{lemma}\label{main.theorem} Let $\Phi\colon T\times\X\tto\Y$ be a set-valued normal integrand, let $\ox\in\dom\Intfset{\Phi}$ satisfy the conditions in \eqref{convex_cond_set}, and let $\oy\in\Intfset{\Phi}(\ox)$. Then we 
have: \vspace*{-0.05in}
\begin{enumerate}[label=\alph*),ref=\ref{main.theorem}-\alph*)]
\item[\bf(i)] \label{Sequential_Regular} For any $\ooy\in\mathcal{S}_{\Phi}(\ox,\oy)$ there exist sequences $\{x_k\}\subset\X$, $\{\x_k\}\subset\Leb^\infty({T},\X)$, $\{\x_k^\ast\}\subset{\Leb}^1({T},\X)$, $\{\y_k\}\subset\Leb^1(T,\Y)$, and $\{\y_k^\ast\}\subset\Leb^\infty(T,\Y)$ such that the following assertions hold:
\begin{enumerate}[label=\alph*),ref=\alph*)] \item[\bf(a)] $\x_k^*(t)\in\Hat{D}^\ast\Phi_t\big(\x_k(t),\y_k(t)\big)\big(\y_k^\ast(t)\big)$ for a.e.\ $t\in T$ and all $k\in\N$. 
\item[\bf(b)] $\|\ox-x_k\|\to 0$, $\|\ox-\x_k\|_\infty\to 0$, $\disp\int_T\|\ooy(t)-\y_k(t)\|\dmu\to 0$, and $\|\y_k^\ast-\oy^\ast\|_\infty\to 0$ as $k\to\infty$. 
\item[\bf(c)] $\disp\int_T\|\x_k^*(t)\|\cdot\|\x_k(t)-x_k\|\dmu\to 0$ and $\disp\int_T\x_k^*(t) \dmu\to\ox^\ast$. 
\end{enumerate}
\item[\bf(ii)]\label{Theo_Regular_coder00} 
Pick $\ooy\in\mathcal{S}_{\Phi}(\ox,\oy)$ and assume that $\Phi$ is integrably quasi-Lipschitzian around $(\ox,\ooy)$. Then we have the inclusion
\begin{equation*}\hspace{-0.1cm}	\Hat{D}^\ast\Intfset{\Phi}(\ox,\oy)(y^\ast)\subseteq\int_T{D}^\ast\Phi_t\big(\ox,\ooy(t)\big)(y^\ast)\dmu\;\mbox{ for all }\;y^\ast\in\Y.
\end{equation*}
\item[\bf(iii)]\label{Theo_lim_basic_Cod00} If $\mathcal{S}_\Phi$ is inner semicompact at $(\ox,\oy)$ and if $\Phi$ is integrably quasi-Lipschitzian around $(\ox,\ooy)$ for all $\ooy \in \mathcal{S}_\Phi(\ox,\oy)$, then
\begin{equation*}
\hspace{-0.1cm}	{D}^\ast\Intfset{\Phi}(\ox,\oy)(y^\ast)\subseteq\bigcup\limits_{\ooy\in\mathcal{S}_{\Phi}(\ox,\oy)}\int_T{D}^\ast\Phi_t\big(\ox,\ooy(t)\big)(y^\ast)\dmu\;\mbox{ for all }\;y^\ast\in\Y.
\end{equation*}
\item[\bf(iv)]\label{cor_Basic_coder} Suppose that $\Phi$ satisfies the   integrable locally Lipschitzian condition \eqref{eq_definition_Int_loc} at $\ox$  and that $ \Intfset{\Phi}(\ox)$ is single-valued. Then $\mathcal{S}_\Phi$ is inner semicompact at $(\ox,\Intfset{\Phi}(\ox))$, and we have the inclusion
\begin{equation*}
{D}^\ast \Intfset{ \Phi }(\ox)(y^\ast)\subseteq \int_T {D }^\ast \Phi_t(\ox)(y^\ast)\dmu.
\end{equation*}
\end{enumerate}
\end{lemma}	

Here is the main result of this section giving us sufficient conditions for the Lipschitz-like property of expected-integral multifunctions \eqref{def:set-valued:Exp}.\vspace*{-0.02in}
	
\begin{theorem}\label{LipLikeEphi}
Let $\Phi: T\times \X \tto \Y$ be a set-valued normal integrand, let $\ox\in \dom \Phi$ under condition \eqref{convex_cond_set}, and let $\oy \in \Intfset{\Phi}(\ox)$.  Assume that\vspace*{-0.05in}
\begin{enumerate}[label=\alph*)]
\item[\bf(a)] \label{LipLikeEphi_casea} either $\Phi$ satisfies the integrable quasi-Lipschitzian condition \eqref{Int_Lips_like_inq} around $(\ox,\ooy)$ for all $\ooy \in \mathcal{S}_\Phi(\ox,\oy)$ and $\mathcal{S}_\Phi$ is inner semicompact at $(\ox,\oy)$,
\item[\bf(b)] \label{LipLikeEphi_caseb} or $\Phi$ satisfies the integrable locally Lipschitzian condition \eqref{eq_definition_Int_loc} at $\ox$.
\end{enumerate}\vspace*{-0.05in}
Then the expected-integral multifunction $\Intfset{\Phi}$ is Lipschitz-like around $(\ox,\oy)$.
\end{theorem}\vspace*{-0.03in}
\begin{proof}
Let us check that in both cases of the theorem the expected-integral multifunction \eqref{def:set-valued:Exp} satisfies the coderivative criterion \eqref{cod-cr} for the Lipschitz-like property around $(\ox,\oy)$. Pick any $x^\ast \in D^\ast \Intfset{\Phi} (\ox,\oy)(0)$. Then in Case~(a) we use (iii) of Lemma~\ref{Theo_lim_basic_Cod00} and find $ \ooy \in S_{\Phi}(\ox, \oy ) $ together with an integrable selection $\x^\ast(t) \in {D }^\ast \Phi_t(\ox,\ooy(t))(0)$ for a.e. $t\in T$ such that
\begin{equation}\label{int}
x^\ast = \int_T \x^\ast(t)\dmu.
\end{equation}
It follows from the coderivative estimate \eqref{Prop61_eq01} of Theorem~\ref{Prop61} that $\x^\ast(t) =0$ for a.e. $t\in T$, and hence $x^\ast=0$ by \eqref{int}. 

In Case~(b), we employ the coderivative limiting representation \eqref{Lim_repr_cod} to find $x_k \to \ox$, $x^\ast_k \to x^\ast$, $y_k \to \oy$, $y_k^\ast \to 0$ with $x^\ast_k \in \Hat D^\ast \Intfset{\Phi}(x_k,y_k)(y_k^\ast)$. Picking any  $\y_k \in  \mathcal{S}(x_k,y_k)$ and applying to each $x^\ast_k$ the results of Lemma~\ref{Sequential_Regular}(i) with the subsequent usage of the diagonal process allow us to choose selections $\y_k(t)\in D^\ast \Phi_t(\x_k(t),\y_k(t))(\y^\ast(t))$ such that 
$$
\| \x_k - \ox\|_\infty,\;\| \y_k - \ooy\|_1,\;\|\y^\ast_k(t) \| \to 0,\;\mbox{ and }\;\int_T \x_k^\ast(t) \dmu \to x^\ast\;\mbox{ as }\;k\to\infty.
$$
Employing further estimate \eqref{Prop61_eq02} from Theorem~\ref{Prop61}, we get
$$
\|\x_k^\ast(t)\| \leq \ell(t)\| \y_k^\ast(t)\|\;\mbox{ for a.e. }\;t\in T\;\mbox{ and large }\;k\in\N.
$$
Passing to the limit as $k\to\infty$ yields $x^\ast=0$ and so completes the proof.
\end{proof}\vspace*{-0.25in}

\section{Sensitivity analysis for stochastic constraint systems}\label{Sensitivity_analysis_parametric_constraint}\vspace*{-0.05in}\setcounter{equation}{0}

In this section we conduct a local sensitivity analysis for the class of {\em parametric stochastic systems} that naturally arise as sets of {\em feasible solutions} to parameterised problems of stochastic optimization and related topics. The section is split into two subsections. The first subsection concerns a general class of stochastic constraint systems, while the second one deals with constraint systems coming from stochastic programming. Our sensitivity analysis addresses deriving efficient conditions for the fulfillment of {\em well-posedness properties} of stochastic constraint systems in the sense of establishing their {\em Lipschitzian stability} and 
{\em metric regularity} by using coderivative evaluations and coderivative characterizations of these well-posedness properties. \vspace*{-0.25in}

\subsection{\bf General stochastic constraint systems}\label{subsec:general constraint} \vspace*{-0.1in}

Given a set-valued normal integrand $\Phi: T\times \X \times\Z \tto \Y$ together with nonempty closed sets $\mathcal{K}\subseteq \Z$ and $\mathcal{O}\subseteq \X\times \Y$, consider the class of general parametric {\em stochastic constraint systems} $F:\X\tto \Z$ described by 
\begin{align}\label{def_mapF}
F(x):=\big\{ z\in \Z\;\big|\;\Intfset{\Phi}(x,z) \cap \mathcal{K}\neq \emptyset, \; (x,z)\in \mathcal{O}\big\},
\end{align} 
where $\Intfset{\Phi}(x,z)$ is a parameterized expected-integral multifunction defined as in \eqref{def:set-valued:Exp}. In the framework of \eqref{def_mapF},   the set $\mathcal{O}$ can be given, e.g., in the form
\begin{equation*}
\mathcal{O}: =\big\{ (x,z)  \in \X\times \Z\;\big|\;(x,z)\in C(t) \text{ for a.e. } t\in T \big\}
\end{equation*}
via some measurable multifunction $C: T\to \X\times\Z$. Integral representations for normal vectors to such sets and their investigation by using new {\em stochastic extremal principles} are presented in \cite{mp1}.\vspace*{0.03in} 

Here is a more specific example of the stochastic constraint systems arising in stochastic programming studied in the next subsection.\vspace*{-0.03in}    
 	
\begin{example}\label{examplePx}
Consider the following parameterized optimization problem:
\begin{align}\label{stoch}\tag{P$_x$}
\min\Intfset{f}(x,z)
\text{ subject to } \Intfset{g_i}(x,z) \leq 0,\; i=1,\ldots,s,\; z\in G(x),
\end{align}
where $f$ and $g_i$, $i=1,\ldots,s$, are normal integrands, and where $G$ is a set-valued mapping. 	Then for a given parameter $x$, the set of {\em feasible solution} to problem \eqref{stoch} can be represented in form  \eqref{def_mapF} with $\mathcal{K}:=\mathbb{R}^n_{-}$,  $\mathcal{O}:=\gph G$, and the integrand  $\Phi(t,x,z):=(g_1(t,x,z),\ldots,g_m(t,x,z))^\top $.
\end{example}\vspace*{-0.05in}

In general, {\em local sensitivity analysis} of parametric systems consists of evaluating a ({\em generalized}) {\em derivative} of the underlying solution map at the point in question and then using this calculation for making conclusions on {\em well-posedness} of the system in questions with respect to small perturbations of the solution point and the nominal parameter. In the case of \eqref{def_mapF}, our generalized derivative is the the (robust) {\em limiting coderivative} \eqref{lcod}, which gives us the {\em coderivative criteria} \eqref{cod-cr} and \eqref{mr} for the (robust) {\em Lipschitz-like} \eqref{lip} and {\em metric regularity} \eqref{metreg} properties, respectively. \vspace*{0.03in}

To proceed, observe that the parametric version of the standing assumptions in \eqref{convex_cond_set} for systems \eqref{def_mapF} at the reference point $(\ox, \oz) \in \gph F$ is formulated as follows: there exist $\rho>0$ and $\kappa \in\Leb^1(T,\R)$ such that   
\begin{equation}\label{convex_cond_set02} 
\begin{aligned}
\Phi_t(x,z) &\text{ is convex for all } (x,z)\in \mathbb{B}_\rho(\ox,\oz) \text{ and all } t\in T_{na},\\
\Phi_t(x,z) &\subseteq \kappa(t) \mathbb{B}_\rho \text{  for all } (x,z) \in \mathbb{B}_\rho(\ox,\oz) \text{ and all } t\in T,
\end{aligned}
\end{equation}	
where $T_{na }$ is the nonatomic part of the measure $\mu$ on $T$.\vspace*{0.05in}

The next theorem provides an efficient evaluation ({\em upper estimate}) of the limiting coderivative in terms of the given data of \eqref{def_mapF}, which is a bridge to the subsequent Lipschitzian stability and metric regularity results.\vspace*{-0.05in}
 	
\begin{theorem}\label{Theo_Sens_general} Let $F: \X \tto \Z$ be given in  \eqref{def_mapF}, and let $(\ox,\oz) \in \gph F$. In addition to \eqref{convex_cond_set02}, impose the following assumptions valid for all $\oy \in  \Intfset{\phi}(\ox, \oz) \cap \mathcal{K}$ and all $ \y \in \mathcal{S}_{\Phi} (\ox,\oz,\oy)$, where $\mathcal{S}_\Phi$ is defined in \eqref{mapping:SPHI} with $\Phi_t=\Phi_t(x,z)$:
\begin{enumerate}[label=\alph*)]
\item[\bf(a)]\label{Theo_Sens_general_a} The mapping $\mathcal{S}_\Phi$ is inner semicompact at $(\ox,\oz,\oy)$.
\item[\bf(b)]\label{Theo_Sens_general_b} The set-valued normal integrand $\Phi$ enjoys the integrable quasi-Lipschitzian property \eqref{Int_Lips_like_inq} around $(\ox,\oz,\ooy)$.
\item[\bf(c)] We have the constraint qualification conditions 
\begin{eqnarray}\label{Theo_Sens_generalQC02}
\begin{array}{ll}
\disp\Big[0\in \int_T D^\ast \Phi_t(\ox,\oz,\y(t)) (y^\ast) \dmu \text{ and } y^\ast \in N(\oy;\mathcal{K})\Big]\Longrightarrow y^\ast=0,\\
\disp\bigcup_{ y^\ast  \in 	N(\oy;\mathcal{K})}\left[\int_T D^\ast \Phi_t(\ox,\oz,\y(t))(y^\ast) \dmu\right]\bigcap\big(- N((\ox,\oz); \mathcal{O})\big)= \{0\}.
\end{array}
\end{eqnarray} 
\end{enumerate}
Then for all $z^\ast\in \Z$ and $x^\ast \in D^\ast F(\ox,\oz) (z^\ast)$ there exist $\oy\in \Intfset{\Phi}(\ox,\oz)\cap \mathcal{K}$, $y^\ast \in N(\oy;\mathcal{K})$, and $ \y \in \mathcal{S}_{ \Phi} (\ox,\oz,\oy)$ such that
\begin{align}\label{coder-CS}
\begin{pmatrix}
x^\ast\\ - z^\ast
\end{pmatrix}
\in\int_T D^\ast \Phi_t(x,z,\y(t)) (y^\ast) \dmu +N\big((\ox,\oz); \mathcal{O}\big).
\end{align}
\end{theorem}
\begin{proof} By \eqref{convex_cond_set02}, it follows from Lemma~\ref{Theo_Regular_coder00}(iii) that for all $\oy \in  \Intfset{\phi}(\ox, \oz)\cap\mathcal{K}$ we have the coderivative upper estimate
\begin{align}\label{Theo_Sens003}
D^\ast \Intfset{\Phi}(\ox,\oz,\oy)(y^\ast)\subseteq \bigcup\limits_{ \y \in \mathcal{S}_{ \Phi} (\ox,\oz,\oy)} \int_T D^\ast \Phi_t(\ox,\oz,\y(t)) (y^\ast)\dmu.
\end{align}
Observe further that for $\gph F:= \Intfset{\Phi}^{-1}\left(\mathcal{K} \right)\cap\mathcal{O})$
the coderivative inclusion  $x^\ast \in D^\ast F(\ox,\oz) (z^\ast)$ reduces to $(x^\ast, -z^\ast) \in N((\ox, \oz); \Intfset{\Phi}^{-1}\left(  \mathcal{K} \right)  \cap \mathcal{O})$. To verify the assertion of the theorem, we concentrate in what follows on deriving an upper estimate for the normal cone to the latter set. It follows from the basic normal cone intersection rule for closed sets in \cite[Theorem~2.16]{m18} that
\begin{align}\label{Theo_Sens008}
N\big((\ox, \oz);\Intfset{\Phi}^{-1}( \mathcal{K}))\cap \mathcal{O}\big)\subseteq N\big((\ox, \oz); \Intfset{\Phi}^{-1}(\mathcal{K})\big) + N\big((\ox, \oz);\mathcal{O}\big)
\end{align}
provided the fulfillment of the qualification condition
\begin{align}\label{Theo_Sens006}
N\big((\ox, \oz);\Intfset{\Phi}^{-1}\left(  \mathcal{K} \right)\big)\cap\big( -N((\ox, \oz);\mathcal{O})\big) =\{ 0\}.
\end{align}
On the other hand, we get from \cite[Corollary~3.13]{m18} that 
\begin{align}\label{Theo_Sens005}
N\big((\ox, \oz);\Intfset{\Phi}^{-1}\left(  \mathcal{K} \right)\big)  \subseteq\bigcup\Bigg[D^\ast \Intfset{\Phi} 	(\ox,\oz,\oy) (y^\ast)\;\Bigg|\:\begin{array}{c}
y^\ast \in N(\oy;\mathcal{K}), \\
 \oy \in \Intfset{\Phi} 	(\ox,\oz)    \cap \mathcal{K}
\end{array}\Bigg]
\end{align}
under the fulfillment of the qualification condition 
\begin{align*}
N(\oy; \mathcal{K})\cap \ker D^\ast \Intfset{\Phi}(\ox,\oz,\oy) =\{ 0\} \text{ for all } \oy \in \Intfset{\Phi}(\ox,\oz)\cap \mathcal{K},
\end{align*}
which readily follows from \eqref{Theo_Sens003} and the first qualification condition in \eqref{Theo_Sens_generalQC02}. Furthermore, we deduce 
from \eqref{Theo_Sens003} and \eqref{Theo_Sens005} that
\begin{align}\label{Theo_Sens007}
N\big((\ox, \oz);\Intfset{\Phi}^{-1}\left( \mathcal{K} \right)\big)  \subseteq \bigcup\Bigg[\int_T D^\ast \Phi_t(x,z,\y(t))(y^\ast) \dmu \;\Bigg|\;\begin{array}{c}
y^\ast \in N(\oy; \mathcal{K}), \\
\oy \in \Intfset{\Phi}(\ox,\oz)\cap \mathcal{K}\\
\y \in\mathcal{S}_{ \Phi} (\ox,\oz,\oy)
\end{array}\Bigg].
\end{align}
Using the latter together with the second qualification condition in \eqref{Theo_Sens_generalQC02} gives us \eqref{Theo_Sens006}. Combining finally
\eqref{Theo_Sens008} and \eqref{Theo_Sens007} verifies \eqref{coder-CS}. 
\end{proof}\vspace*{-0.05in}

As direct consequences of Theorem~\ref{Theo_Sens_general} and the coderivative characterizations \eqref{cod-cr} and \eqref{mr}, we get sufficient conditions for the Lipschitz-like and metric regularity properties of the stochastic constraint system  \eqref{def_mapF} expressed via the limiting normal cone and coderivative of the system data.\vspace*{-0.05in}

\begin{corollary}\label{Lipschitz_Like_F}
In the setting of Theorem~{\rm\ref{Theo_Sens_general}}, take arbitrary vectors
$y^\ast \in N(\oy;\mathcal{K})$, $\oy \in\Intfset{\Phi}(\ox,\oz)\cap \mathcal{K}$,  and $\y \in \mathcal{S}_{ \Phi} (\ox,\oz,\oy)$. The following assertions hold:
\begin{enumerate}
\item[\bf(i)] Assume the fulfillment of the implication 
\begin{equation}\label{Lipschitz_Like_F01}
\bigg[\begin{pmatrix}
x^\ast\\0
\end{pmatrix} \in\int_T D^\ast \Phi_t\big(\ox,\oz,\y(t)\big) (y^\ast)\dmu+N\big((\ox,\oz);\mathcal{O}\big)\bigg]\Longrightarrow x^\ast=0.
\end{equation}
Then the mapping $F$ from \eqref{def_mapF} is Lipschitz-like around $(\ox,\oz)$.\vspace*{0.03in}
\item[\bf(ii)] If the complemented implication
\begin{align}\label{metreg_F01}
\bigg[\begin{pmatrix}
0\\z^*
\end{pmatrix} \in\int_T D^\ast \Phi_t(\ox,\oz,\y(t)) (y^\ast)\dmu+N\big((\ox,\oz);\mathcal{O}\big)\bigg]\Longrightarrow z^\ast=0
\end{align}
is satisfied, then  $F$ is metrically regular around $(\ox,\oz)$.
\end{enumerate}
\end{corollary}\vspace*{-0.05in}
\begin{proof} Observe first that the imposed general assumptions ensure that the mapping $F$ from \eqref{def_mapF} is closed-graph around $(\ox,\oz)$. To verify (i), we need now checking by the coderivative criterion \eqref{cod-cr} that $ D^\ast F(\ox,\oz)(0)=\{0\}$ for $F$ in \eqref{def_mapF}. Pick any $x^\ast \in D^\ast F(\ox,\oz)(0)$ and then deduce from Theorem~\ref{Theo_Sens_general} that there exist $\oy \in\Intfset{\Phi}(\ox,\oz)\cap\mathcal{K}	$, 	$y^\ast \in N(\oy;\mathcal{K})$, and $ \y \in \mathcal{S}_{ \Phi}(\ox,\oz,\oy)$ such that
\begin{align*}
\begin{pmatrix}
x^\ast\\0
\end{pmatrix} \in\int_T D^\ast \Phi_t(x,z,\y(t)) (y^\ast) \dmu + 	N((\ox,\oz); \mathcal{O}).
\end{align*}
Employing \eqref{Lipschitz_Like_F01} tells us that $x^\ast =0$, which verifies (i). The proof of (ii) is similar by using the coderivative evaluation from Theorem~\ref{Theo_Sens_general} and the coderivative characterization of metric regularity given in \eqref{mr}.
\end{proof}\vspace*{-0.05in}

The next corollary of Theorem~\ref{Theo_Sens_general} addresses the case where $\Phi$ is single-valued and locally Lipschitzian around the reference point.\vspace*{-0.05in}
	
\begin{corollary}\label{Corollary_Cod_F01} For $F$ in \eqref{def_mapF}, assume that $\Phi$ is single-valued and locally Lipschitzian around $(\ox,\oz)$ and 
the following qualification conditions are satisfied:
\begin{align}\label{Cor_Sens_generalQC01}
\bigcup	\left[\int_T \partial \langle y^\ast, \Phi_t \rangle (\ox,\oz) \dmu\Bigg|\,y^\ast\in N\big(\Intfset{\Phi}(\ox,\oz);\mathcal{K}\big)\right]\bigcap \big(- N((\ox,\oz);\mathcal{O})\big)=\{ 0\},
\end{align}
\begin{align}\label{Cor_Sens_generalQC02} 
N\big(\Intfset{\Phi}(\ox,\oz);\mathcal{K}\big)\cap\ker\left[  \int_T\partial  \langle \cdot,  \Phi_t \rangle (\ox,\oz)\dmu\right] =\{0\}.
\end{align}
Then for all $z^\ast\in \Z$ we have the coderivative upper estimate 
\begin{align*}
D^\ast F(\ox,\oz) (z^\ast) \subseteq\Bigg\{ x^\ast\;\Bigg|\begin{array}{c}
\text{ there exist } y^\ast \in N(\Intfset{\Phi}(\ox,\oz);\mathcal{K}) \text{ such that } \\
\begin{pmatrix}x^\ast\\-z^\ast\end{pmatrix}\in \displaystyle \int_T \partial \langle y^\ast,\Phi_t \rangle (\ox,\oz) \dmu +N\big((\ox,\oz); \mathcal{O}\big) 
\end{array}\Bigg\}.\end{align*}
If in addition for any $y^\ast \in N(\Intfset{\Phi}(\ox,\oz);\mathcal{K})$ the implication
\begin{align}\label{cor-liplike}
\Bigg[\begin{pmatrix}
x^\ast\\ 0
\end{pmatrix}\in\displaystyle \int_T\partial \langle y^\ast,  \Phi_t \rangle (\ox,\oz)\dmu + 	N\big((\ox,\oz); \mathcal{O}\big)\Bigg]\Longrightarrow x^\ast=0
\end{align}
holds, then the mapping $F$ is Lipschitz-like around $(\ox,\oz)$. The replacement of \eqref{cor-liplike} by the complemented implication
\begin{align*}
\Bigg[\begin{pmatrix}
0\\z^*
\end{pmatrix}\in\displaystyle \int_T\partial \langle y^\ast,  \Phi_t \rangle (\ox,\oz)\dmu + 	N\big((\ox,\oz); \mathcal{O}\big)\Bigg]\Longrightarrow z^\ast=0
\end{align*}
ensures the metric regularity of $F$ around $(\ox,\oz)$.
\end{corollary}\vspace*{-0.05in}
\begin{proof} It is easy to see that, due to the assumed local single-valuedness and Lipschitz continuity of $\Phi$, assumptions (a) and (b) of Theorem~\ref{Theo_Sens_general} are satisfied. The qualification conditions \eqref{Cor_Sens_generalQC01} and \eqref{Cor_Sens_generalQC02} clearly yield those imposed in 
\eqref{Theo_Sens_generalQC02} by using the scalarization formula \eqref{scal}. 
The latter formula allows us to deduce the claimed coderivative inclusion from \eqref{coder-CS} of Theorem~\ref{Theo_Sens_general}, while the Lipschitz-like and metric regularity properties of $F$ follows from Corollary~\ref{Lipschitz_Like_F} under the imposed additional assumptions. 
\end{proof}\vspace*{-0.05in}

Let us present a direct consequence of Corollary~\ref{Corollary_Cod_F01} addressing stochastic problems with only {\em equality constraints}.\vspace*{-0.02in}

\begin{corollary}
In the setting of Corollary~{\rm\ref{Corollary_Cod_F01}}, consider the system 
$$
F(x):=\big\{ z\in \Z\;\big|\;\Intfset{\Phi}(x,z) =0\big\}$$ under the qualification condition
\begin{align*}
\ker \left[ \int_T D^\ast \Phi_t(\ox,\oz) (\cdot )\dmu\right]=\{0\}.
\end{align*}
Then for all $z^\ast\in \Z$  we have the coderivative upper estimate
\begin{align*}
D^\ast F(\ox,\oz) (z^\ast) \subseteq\Bigg\{ x^\ast \in \X\;\Bigg|\;\begin{array}{c}
\text{ there exits } y^\ast \in \Y \text{such that} \\
\begin{pmatrix}
x^\ast\\-z^\ast
\end{pmatrix}\in  \hspace{-0,1cm} \displaystyle \int_T D^\ast \Phi_t(\ox,\oz) (y^\ast) \dmu \end{array}\Bigg\}.
\end{align*}
The additional fulfillment of the implication
\begin{align*}
\Bigg[\begin{pmatrix}
x^\ast\\0
\end{pmatrix}\in\displaystyle \int_T D^\ast \Phi_t(\ox,\oz) (y^\ast)\dmu \Bigg]\Longrightarrow x^\ast=0\;\mbox{ as }\;y^*\in\Y
\end{align*}
ensures that $F$ is Lipschitz-like around $(\ox,\oz)$, while the implication
\begin{align*}
\Bigg[\begin{pmatrix}
z^\ast\\0
\end{pmatrix}\in\displaystyle \int_T D^\ast \Phi_t(\ox,\oz) (y^\ast)\dmu \Bigg]\Longrightarrow z^\ast=0,\quad y^*\in\Y,
\end{align*}
yields the metric regularity of $F$ around this point.
\end{corollary}
\begin{proof} It follows from Corollary~\ref{Corollary_Cod_F01} with  $\mathcal{O}=\X\times \Z$ and $\mathcal{K}=\{0\}$. 
\end{proof}
\vspace*{-0.35in}

\subsection{\bf Constraint systems in stochastic programming}\label{subsec:constraint-program}\vspace*{-0.05in}

In this subsection we provide a coderivative-based local sensitivity analysis for feasible solution maps in problems of {\em stochastic programming} of 
type \eqref{stoch} that are described as follows:
\begin{align}\label{F_Inq1}
F(x):=\big\{ z\in \Z\;\big|\;\Intfset{\Phi}(x,z)\leq 0,\;z\in G(x)\big\},
\end{align}
where $\Phi(t,x,z):=(g_1(t,x,z), \ldots, g_m(t,x,z))^\top $ with some normal integrands $g_1, \ldots,g_m: \X\times \Z \to \R\cup\{ +\infty\}$ and a set-valued mapping $G:\X\tto \Z$.\vspace*{0.03in}

Our main goal in this subsection is to evaluate the coderivative \eqref{lcod} of \eqref{F_Inq1}, which leads us as in Subsection~\ref{subsec:general constraint} us sufficient conditions the Lipschitz-like and metric regularity properties of this mapping due to their coderivative characterizations. For brevity, we address here only the Lipschitz-like property of  the constraint system \eqref{F_Inq1} and its specifications.\vspace*{0.03in}

First we present the following proposition, which is a consequence of Theorem~\ref{Theo_Sens_general} and Corollary~\ref{Lipschitz_Like_F} for the case of feasible solution maps \eqref{F_Inq1}.\vspace*{-0.05in}

\begin{proposition}\label{Corollary_Cod_F012} Let $F$ be given in  \eqref{F_Inq1}, where $g_i$, $i=1,\ldots,m$, are locally Lipschitzian around $(\ox,\oz)\in\gph F$, and where the graph of $G$ is locally closed around this point. Suppose that for all $(-x^*,z^*)\in\gph D^*G(\ox,\oz)$ we have the two constraint  qualification conditions:
\begin{align}\label{Cor_SensQC00} \Bigg[\begin{pmatrix}
x^\ast\\ z^\ast 
\end{pmatrix} &\in \bigcup\left(\int_T \partial \langle  y^\ast,\Phi_t \rangle (\ox,\oz) \dmu\,\Bigg|\,y^\ast \in N(\Intfset{\Phi}\big(\ox,\oz);\Y_{-}\big)\right)\Bigg]\\ \nonumber
&\Longrightarrow \begin{pmatrix}
x^\ast\\ z^\ast 
\end{pmatrix}  =\begin{pmatrix}
0\\ 0 
\end{pmatrix},
\end{align}
\begin{align}\label{Cor_SensQC01}\Bigg[\begin{pmatrix}
0 \\ 0 
\end{pmatrix} \in   \int_T \partial \langle y^\ast,  \Phi_t \rangle (\ox,\oz)   \dmu,\;y^\ast \in N\big(\Intfset{\Phi}(\ox,\oz);\Y_{-}\big) \Bigg]\Longrightarrow y^\ast=0.
\end{align}
Then for all $z^\ast\in \Z$  the coderivative upper estimate
\begin{align}\label{FormulaCorollary45}
D^\ast F(\ox,\oz) (z^\ast) \subseteq\Bigg\{ x^\ast \in \X\;\Bigg|\begin{array}{c}
\text{ there exist }   y^\ast \in N\big(\Intfset{\Phi}(\ox,\oz);\Y_{-}\big)\\ \text{ and  } -u^\ast \in D^\ast G(\ox,\oz)(v^\ast)   \text{ with}  \\
\begin{pmatrix}
x^\ast+u^\ast \\-z^\ast -v^\ast 
\end{pmatrix}\in  \displaystyle \int_T \partial \langle y^\ast,  \Phi_t \rangle (\ox,\oz)  \dmu \end{array}\Bigg\}
\end{align}
holds. If in addition the following implication  
\begin{align}\label{Cor_SensQC03}
\Bigg[\begin{pmatrix}
x^\ast+u^\ast \\-v^\ast 
\end{pmatrix} \in\displaystyle \int_T  \partial \langle y^\ast,  \Phi_t \rangle (\ox,\oz)  \dmu\Bigg]  \Longrightarrow x^\ast=0
\end{align}
is satisfied for all  $-u^\ast \in D^\ast G(\ox,\oz)(v^\ast)$ and if also $y^\ast \in N(\Intfset{\Phi}(\ox,\oz);\Y_{-})$ holds for all  $-u^\ast \in D^\ast G(\ox,\oz)(v^\ast) $  and  $y^\ast \in N(\Intfset{\Phi}(\ox,\oz);\Y_{-})$,
then the feasible solution map $F$ is Lipschitz-like around the reference point $(\ox,\oz)$.
\end{proposition}\vspace*{-0.05in}
\begin{proof} This follows from Corollary~{\rm\ref{Corollary_Cod_F01}} with the sets $\mathcal{K}:=\mathbb{R}^n_{-}$,  $\mathcal{O}:=\gph G$, and the normal integrand  $\Phi(t,x,z):=(g_1(t,x,z),\ldots,g_m(t,x,z))^\top $.
\end{proof}\vspace*{-0.1in}

Specifying further assumptions on the initial data of the stochastic
program \eqref{stoch} from Example~\ref{examplePx}, we present now efficient conditions ensuring the fulfillment of the major qualification condition \eqref{Cor_SensQC01} in Proposition~\ref{Corollary_Cod_F012}.\vspace*{-0.05in}

\begin{example}\label{ex-revisited} Under the notation of Example~\ref{examplePx}, consider the feasible solution map $F$ associated with  \eqref{stoch} and suppose that the integrands $g_i$ are locally Lipschitzian around $(\ox,\oz)$ with $\partial g_i (t,\ox,\oz) \subseteq \X_+ \times \Z_+ \backslash \{ 0\}$ for a.e. $t\in T$. Then we claim that the qualification condition  \eqref{Cor_SensQC01} holds. To check this, pick any $y^\ast\in  N(\Intfset{\Phi}(\ox,\oz);\Y_+)$ satisfying the condition 
\begin{equation}\label{exa}
0 \in  \int_T \partial \langle y^\ast,\Phi_t \rangle (\ox,\oz)\dmu
\end{equation}
with $\Phi(t,x,z)=(g_1(t,x,z), \ldots,g_m(t,x,z))^\top $. It is easy to see that all components of $y^\ast=(y^\ast_1, y^\ast_2, \ldots, y^\ast_m)$ are nonnegative. Using the subdifferential sum rule from \cite[Theorem~2.19]{m18}, we get that 
$$ 
\partial \langle y^\ast,  \Phi_t \rangle (\ox,\oz)  \subseteq \sum_{i=1}^m y_i^\ast \partial g_i(t,\ox,\oz) \text{ for a.e. } t\in T,
$$
which yields by \eqref{exa} the inclusion
$$
0 \in\disp\sum_{i=1}^m y_i^\ast \int_{T} \partial g_i(t,\ox,\oz) \dmu.
$$ 
Since $\partial g_i (t,\ox,\oz) \subseteq \X_+ \times \Z_+ \backslash \{ 0\}$, we verify the claim that $y^\ast=0$.
\end{example}\vspace*{-0.05in}

For the next result we need to recall another subdifferential construction for extended-real-valued functions $f\colon\X\to\oR$ finite at $x$ that is known as the {\em regular subdifferential} of $f$ at $x$ and is defined by
\begin{equation}\label{rsub}
\Hat\partial f(x):=\Big\{x^*\in\X\;\Big|\;\liminf_{u\to x}\frac{f(u)-f(x)-\la x^*,u-x\ra}{\|u-x\|}\ge 0\Big\}.
\end{equation}
If $f$ is l.s.c.\ around $x$, then the subdifferential \eqref{sub} admits the limiting representation via \eqref{pk}, where `$u\st{f}{\to}x$' indicates that $u\to x$ with $f(u)\to f(x)$:
\begin{equation}\label{lim-reg}
\partial f(x)=\Limsup_{u\st{f}{\to}x}\Hat\partial f(u),
\end{equation}

The following lemma establishes relationships between {\em full} and {\em partial subdifferentials} for functions of two variables. This new result is certainly of its own interest, while it is needed in our subsequent coderivative analysis of the feasible solution map \eqref{F_Inq1} in stochastic programming. \vspace*{-0.05in}
 
\begin{lemma}\label{LemmaPartialSub}
Let $g: \X\times \Z \to \R\cup \{ +\infty\}$ be an l.s.c.\ function such that $g(\ox,\cdot)$ is convex. Then we have $z^\ast \in {\partial}_z g(\ox,\oz)$ whenever $ (x^\ast,z^\ast ) \in \Hat{\partial} g(\ox,\oz) $. If in addition $g(x',\cdot)$ is convex and $g(\cdot, z')$ is continuous at $\bar x$ for all $(x',z')$ sufficiently close to $(\ox,\oz)$, then $z^\ast\in{\partial}_z g(\ox,\oz)$ whenever $ (x^\ast,z^\ast ) \in {\partial} g(\ox,\oz) $. 
\end{lemma}\vspace*{-0.05in}
\begin{proof} 
Pick any $(x^\ast,z^\ast )\in \Hat{\partial} g(\ox,\oz)$ and fix an arbitrary number $\epsilon>0$. Using definition \eqref{rsub}, find $\eta >0$ such that for all $(x,z) \in \mathbb{B}_\eta (\ox,\oz)$ we have
\begin{align*}
\langle x^\ast , x- \ox \rangle + \langle z^\ast, z- \oz\rangle\leq g(x,z) - g(\ox,\oz) + \epsilon( \| x-\ox\| + \| z-\oz\| ).
\end{align*}
Fix $z'\in \Z$ and employ the above inequality with  $x:=\ox$ and $z: = \oz + \alpha( z'-\oz )$ for $\alpha>0$ sufficiently small. Then we get
\begin{align*}
\alpha  \langle z^\ast, z'- \oz\rangle &\leq g(\ox,\oz + \alpha ( z'-\oz )) - g(\ox,\oz) + \epsilon\alpha (   \| z'-\oz\| )\\
&\leq \alpha \left( g(\ox,z')  - g(\ox,\oz) + \epsilon\| z' - z\| \right),
\end{align*}
where the latter inequality is due to the convexity of $g$ with respect to $z$. Therefore, for all $\epsilon>0$ and $z'\in \Z$ we obtain that
\begin{align*}
\langle z^\ast, z'- \oz\rangle 
&\leq g(\ox,z')  - g(\ox,\oz) + \epsilon\| z' - z\|.
\end{align*}
Since this holds for each $\epsilon>0$, it tells us that $z^\ast \in \Hat{\partial}_z g(\ox,\oz) $, and hence
$$
(x^\ast,z^\ast ) \in \Hat{\partial} g(\ox,\oz)\Longrightarrow z^\ast \in  \Hat{\partial}_z g(\ox,\oz).
$$  
To prove the last statement of the lemma, take any $(x^\ast,z^\ast ) \in  {\partial} g(\ox,\oz)$ and then find by \eqref{lim-reg} sequences $(x^\ast_k ,z_k^\ast) \in \Hat{\partial } g(x_k,z_k)$ with $(x_k,z_k,x^\ast_k ,z_k^\ast) \overset{g}{\to} ( \ox,\oz,x^\ast,z^\ast)$ as $k\to\infty$. This gives us for all $k \in \N$ that
\begin{align}\label{Ineqlemma46}
\langle z_k^\ast, z'- z_k\rangle\le (x_k,z')-g(x_k,z_k)\text{ whenever } z'\in \Z.
\end{align}
Choose $ \gamma >0 $ such that for all $z'\in \mathbb{B}_\gamma (\oz)$ the function $g(\cdot, z')$ is continuous at $\ox$. Passing to the limit as $k\to\infty$ in \eqref{Ineqlemma46}, we conclude that  
\begin{align*} 
\langle z^\ast, z'- \oz \rangle  
\le g(\ox,z')  - g(\ox,\oz) \text{ for all } z'\in \mathbb{B}_\gamma (\oz),
\end{align*}
which ensures by the assumed convexity of $g(\ox,\cdot)$ that $z^\ast \in \partial_z g(\ox,\oz)$. This therefore completes the proof of the lemma.
\end{proof}\vspace*{-0.1in}

Now we are ready to obtain {\em explicit conditions} on the initial data of \eqref{F_Inq1} supporting  the coderivative evaluation and the Lipschitz-like property of the feasible solution map associated with the stochastic program \eqref{stoch}.  
\vspace*{-0.05in}

\begin{theorem}\label{Corollary46}
Let $F$ be given in  \eqref{F_Inq1}, where the functions $g_i$ are locally Lipschitzian around $(\ox,\oz)\in\gph F$ being convex with respect to the second variable, and where $G$ is of closed graph. Assume also that:\vspace*{-0.05in}
\begin{enumerate}[label=\alph*)]
\item[\bf(a)] The set $G(\ox)$ is convex, while the set $\gph G$ is normally regular at $(\ox,\oz)$.
\item[\bf(b)] There exits a vector $z_0 \in G(\ox) \backslash \{ \oz\}$ such that we have $\Intfset{g_i}(\ox,z_0) < 0$ for all $i=1,\ldots, m$ and that $\Intfset{g_i}(\ox,\oz)=0$.
\end{enumerate}\vspace*{-0.05in}
Then the coderivative upper estimate \eqref{FormulaCorollary45} holds. If in addition $G$ is Lipschitz-like around $(\ox,\oz)$, then $F$ enjoys this property around $(\ox,\oz)$.
\end{theorem}\vspace*{-0.1in}
\begin{proof} Consider the set of active inequality constraint indices $I(\ox,\oz):=\{i\in \{1,\ldots, m\}\;|\;\Intfset{g_i}(\ox,\oz)=0\}$ for $F$ at $(\ox,\oz)$ and take $-x^\ast \in D^\ast G(\ox,\oz)(z^\ast) $ and $y^\ast=(y_1^\ast, \ldots,y_m^\ast)^\top \in N(\Intfset{\Phi}(\ox,\oz);\Y_{-})$ such that 
\begin{equation}\label{cod-int1}
\begin{pmatrix}
x^\ast\\ z^\ast 
\end{pmatrix}\in\int_T \partial \langle  y^\ast,  \Phi_t \rangle (\ox,\oz) \dmu.
\end{equation}
Due to Lemma~\ref{lemma_measurability_reg_sub} and the measurable selection theorem discussed in Subsection~2.2, as well as by the uniform boundedness  subdifferential mappings generated by locally Lipschitzian functions, there exists an integrable selection $ (\x^\ast(t), \z^\ast (t) )^\top 
\in\partial \langle  y^\ast, \Phi_t \rangle (\ox,\oz) $ for a.e. $t\in T$ represented by
$$
(x^\ast, z^\ast)^\top = \int_T\big(\x^\ast(t),\z^\ast (t)\big)^\top \dmu.
$$
Using the limiting subdifferential sum rule and employing again the measurable selection theorem allow us to find integrable functions $(\v_i^\ast(t),\w_i^\ast(t))^\top\in\partial g_i(t,\ox,\oz)$ for a.e. $t\in T$ and all $i=1,\ldots, m$ satisfying the equality
\begin{align*}
(\x^\ast(t),\z^\ast (t) ) = \sum_{i=1}^{m } y_i^\ast  (\v_i^\ast(t),\w_i^\ast(t))\text{ for a.e. } t\in T.
\end{align*}
Lemma~\ref{LemmaPartialSub} tells us that $\w^\ast_i(t)\in \partial_z g_i(\ox,\oz)$ whenever $i=1,\ldots,m$. Furthermore, we have that $y^\ast_i = 0$ for all $i\notin I(\ox,\oz)$. Picking now any $i\in I(\ox,\oz)$ and taking into account the assumptions in (b) of the theorem ensure that
\begin{align}\label{SLATER}
\int_{ T  } \langle \w_i^\ast(t), z_0 -\oz  \rangle \dmu \leq  \Intfset{g_i}(z_0,\ox) <0,\quad i\in I(\ox,\oz).
\end{align}
On the other hand, we apply Lemma~\ref{LemmaPartialSub} to the indicator function of the set $\gph G$ by using the assumptions in (a) of the theorem. This gives us $-z^\ast\in N(\oz;G(\ox))$ and therefore ensures that
\begin{align*}
0\leq \langle z^\ast  , z_0 - \oz \rangle =   \sum_{i=1}^my_i^\ast \int_{ T  } \langle \w_i^\ast(t), z_0 -\oz  \rangle\mu(dt) \leq   \sum_{i\in I(\ox,\oz)}	 y_i^\ast  \Intfset{g_i}(z_0,\ox).  
\end{align*} 
It follows from \eqref{SLATER} that $y^\ast_i =0$ for all $i\in I(\ox,\oz)$, and consequently  $y^\ast=0$ for the whole vector $y^*$, which implies by \eqref{cod-int1} that $(x^\ast, z^\ast)=(0,0)$. This allows us to conclude that the qualification conditions \eqref{Cor_SensQC00} and \eqref{Cor_SensQC01} are satisfied, and thus we deduce the coderivative estimate \eqref{FormulaCorollary45} from Proposition~\ref{Corollary_Cod_F012}.

To verify next the fulfillment of implication \eqref{Cor_SensQC03}, pick any vectors $y^\ast\in N(\Intfset{\Phi}(\ox,\oz);\Y_{-})$ and $-u^\ast\in D^\ast G(\ox,\oz)(v^\ast) $ satisfying \begin{align*} 
\begin{pmatrix}	x^\ast+u^\ast \\-v^\ast 
\end{pmatrix}\in \displaystyle \int_T  \partial \langle y^\ast,  \Phi_t \rangle (\ox,\oz)\dmu.
\end{align*}
The above arguments readily show that $y^\ast=0$, and therefore $x^\ast + u^\ast=0$ and $v^\ast=0$. Thus the imposed assumption that $G$ is Lipschitz-like around $(\ox,\oz)$ yields by the coderivative criterion \eqref{cod-cr} that $u^\ast=0$ and consequently $x^\ast=0$. The claimed assertion on the Lipschitz-like property of $F$ around $(\ox,\oz)$ is due to Proposition~\ref{Corollary_Cod_F012}. This completes the proof of the theorem.\end{proof}\vspace*{-0.1in}

To conclude this section, we consider stochastic constraint systems  \eqref{F_Inq1}, which are generated by {\em polyhedral} set-valued mappings $G$, where the coderivative evaluation and Lipschitzian stability can be deduced from Theorem~\ref{Corollary46}.\vspace*{-0.05in}

\begin{example} Given integrable functions $a_i : T  \to \Z $ and normal integrands $b_i : T\times \X \to \R$ for $i=1,\ldots,m$,  define the constraint mapping
\begin{equation}\label{semi}
F(x):=\bigg\{  z  \in \Z\;\bigg|\;\Big\langle \int_T a_i(t)\dmu ,z\Big\rangle \leq \Intfset{b_i}(x) \text{ and } z\in G(x)\bigg\},
\end{equation}
where $G:\X\tto \Z$ is a polyhedral convex multifunction. Constraint mappings of type \eqref{semi} are known as  {\em semilinear stochastic parametric systems}. Assume that there exists a vector $z_0 \in G(\ox)\backslash \{ \oz\}$ such that
$$
\Big\langle \int_T a_i(t)\dmu ,z_0\Big\rangle < \Intfset{b_i}(\ox)\;\mbox{ for all }\;i=1,\ldots,m.
$$
Recalling the classical result from \cite{ww} (see also \cite[Theorem~3C.3]{dr}) establishing the Lipschitz continuity of polyhedral convex multifunctions, we conclude from Theorem~\ref{Corollary46} that the semilinear constraint mapping \eqref{semi} is Lipschitz-like around $(\ox,\oz)$. Furthermore, it follows from the above theorem that for any $x^\ast \in D^\ast F(\ox,\oz)(-z^\ast)$ there exist multipliers $\lambda_i \geq 0$ for $i=1,\ldots,m$ as well as $u^\ast \in D^\ast G(\ox,\oz)(-v^\ast)$ and an integrable selection $\x^\ast(t) \in \partial b_t(\ox)$  such that 
\begin{align*}
0&=\lambda_i\left(\bigg\langle \int_T a_i(t)\dmu ,z\bigg\rangle - \Intfset{b_i}(x)\right) \text{ for all } i=1,\ldots,m,\\
z^\ast + v^\ast &=\sum_{ i=1}^m \lambda_i \int_T a_i(t)\dmu,\;\mbox{ and }\;
x^\ast + u^\ast = \sum_{ i=1}^m \lambda_i\int_T \x^\ast(t)\dmu.
\end{align*}
\end{example}\vspace*{-0.25in}

\section{Sensitivity analysis for stochastic variational systems}\label{Sensitivity_variational}\vspace*{-0.1in}\setcounter{equation}{0}

This section is devoted to a local sensitivity analysis of the class of parametric stochastic systems given by
\begin{align}\label{def_mapS}
S(x):=\big\{z \in\Z\;\big|\;0\in\Intfset{\Phi}(x,z)+G(x,z)\big\},
\end{align}
where $\Phi: T\times \X \times\Z \tto \Y$ is a set-valued normal integrand, and where  $G: \X\times \Z\tto \Y$  is a closed-graph multifunction. Similarly  to the deterministic case \cite{m06}, we label systems of type \eqref{def_mapS} as {\em stochastic variational systems}.  This name comes from the fact that in many situations the framework of \eqref{def_mapS} describes optimality conditions in problems of stochastic optimization, variational inequality, complementarity systems, etc. Let us illustrate this by the following rather general example.\vspace*{-0.05in}

\begin{example}\label{example_VI}
Consider a  probability space $(\Omega, \mathcal{A},\mathbb{P})$, a set-valued normal integrand $\Phi: \Omega\times \X\times\Z\tto\Z$, and closed convex set $\emp\ne K\subseteq\Z$. Given $(\omega,x)\in\Omega\times \X$,  define the parameterized {\em stochastic variational inequality} by
\begin{equation}\label{VI_example}
\text{find }z \in K\;\text{ such that }\;\langle \Phi_\omega(x, z) , y-z \rangle \geq 0\;\mbox{ for all }\;y \in K.
\end{equation}
The set of solutions in the expected value formulation of \eqref{VI_example} can be represented in the form of \eqref{def_mapS} as
\begin{align}\label{svi}
S(x):=\big\{z\in \Z\;\big|\;0\in\Intfset{\Phi}(x,z)+N(z;K)\big\},
\end{align}
i.e., with $G(x,z):=N(z;K)$. The reader can find more details about such stochastic systems in\cite{cwz,gor} with the discussions and references therein.
\end{example}\vspace*{-0.03in}

Our goal here is to provide a coderivative-based sensitivity analysis of stochastic variational systems of type \eqref{def_mapS} with establishing efficient upper estimates of the limiting coderivative of $S(x)$ and applying  these estimates to deriving sufficient conditions for Lipschitzian stability and metric regularity of such systems. Both of these well-posedness properties can be simultaneously derived from coderivative calculations for general stochastic systems \eqref{def_mapS} similarly to the device of Section~\ref{Sensitivity_analysis_parametric_constraint}. However, there exists a  {\em significant difference} between constraint and variational systems: in the major cases of \eqref{def_mapS} where $G(x,z)$ is given via subdifferential and/or normal cone mappings (like, e.g., in the case of variational inequalities \eqref{svi}) the {\em metric regularity fails}, i.e., the coderivative-based sufficient conditions for the metric regularity of such systems are not satisfied. The reader can find more details for deterministic variational systems of this type in \cite{m08} and \cite[Section~3.3]{m18}. Having this in mind, we establish below the corresponding results for both Lipschitz-like and metric regularity property of general systems \eqref{def_mapS}, while address only Lipschitzian stability of stochastic variational inequalities of type \eqref{svi}.\vspace*{0.03in}

To proceed, let us introduce the \emph{adjoint generalized equation} to \eqref{def_mapS} at a reference point $(\ox,\oz,\oy)$ defined as
\begin{equation}\label{adj_gen_eq}
0 \in \bigcup\limits_{ \y \in \mathcal{S}_{ \Phi} (\ox,\oz,\oy)  } \int_T D^\ast \Phi_t\big(\ox,\oz,\y(t)\big) (y^\ast) \dmu + D^\ast G(\ox,\oz,-\oy)(y^\ast).
\end{equation}\vspace*{-0.1in}

\begin{theorem}\label{Theorem:Parametric:main}
Let  $S: \X \tto \Z$ be given in  \eqref{def_mapS} with  $\Phi: T\times \X\times \Z \tto \Y$ and $G: \X\times \Z\tto \Y$, and let $\mathcal{S}_{ \Phi}$ be taken from \eqref{mapping:SPHI} with $\Phi_t=\Phi_t(x,z)$. Pick $(\ox,\oz) \in \gph S$ and suppose that $\Phi$ satisfies \eqref{convex_cond_set02}, and that for all $\oy \in  -G(\ox,\oz)\cap \Intfset{\Phi}(\ox,\oz)$ and all  $ \y \in \mathcal{S}_{ \Phi} (\ox,\oz,\oy)$ we have the conditions:
\vspace*{-0.02in}
\begin{enumerate}[label=\alph*)]
\item[\bf(a)] \label{Theo_Para_general_a} The mapping $\mathcal{S}_\Phi$ is inner semicompact at $(\ox,\oz,\oy)$.
\item[\bf(b)]\label{Theo_Para_general_b}  $\Phi$ possesses the integrable quasi-Lipschitz condition \eqref{Int_Lips_like_inq} around $(\ox,\oz,\ooy)$.
\end{enumerate}\vspace*{-0.02in}
Then for every $z^\ast \in \Z$ and $x^\ast \in  D^\ast S(\ox,\oz) (z^\ast)$ there exist  $y^\ast \in \Y$ and $\oy \in -G(\ox,\oz)\cap \Intfset{\Phi}(\ox,\oz)$  such that
\begin{align*}
\left(\begin{array}{c}
x^\ast \\ -z^\ast 
\end{array} \right) &\in  D^\ast G(\ox,\oz,-\oy)(y^\ast) +  \bigcup\limits_{ \y \in \mathcal{S}_{ \Phi} (\ox,\oz,\oy)  } \int_T D^\ast \Phi_t(\ox,\oz,\y(t)) (y^\ast) \dmu 
\end{align*}
provided that for all $\oy \in -G(\ox,\oz)\cap \Intfset{\Phi} (\ox,\oz)$ the adjoint generalized equation \eqref{adj_gen_eq} admits only the trivial
solution $y^\ast=0$.
\end{theorem}\vspace*{-0.05in}
\begin{proof}
Observe that $\gph S=H^{-1}(\gph \Intfset{\Phi } )$, where $H(x,z):=(x,y, -G(x,z))$. For $w:=(\ox,\oz,\oy)\in H(\ox,\oz)$ and $w^\ast: =(x^\ast,z^\ast,y^\ast)$ we easily get that
\begin{align}\label{codH}
D^\ast H((\ox,\oy),w)(w^\ast)&= (x^\ast,0) + (0,z^\ast) +D^\ast(-g)(\ox,\oz,-\oy)(y^\ast)\\\nonumber
&=  \left( \begin{array}{c}
x^\ast \\ z^\ast 
\end{array} \right)+D^\ast g(\ox,\oz,-\oy)( -y^\ast).
\end{align}
Using \eqref{codH} and calculating the normal cone to $H^{-1}(\gph \Intfset{\Phi })$ by \cite[Corollary~3.13]{m18} tell us that $N((\ox, \oz); \gph S ) $ is contained in the set 
\begin{align}\label{Theo_Par005}
\bigcup \hspace{-0.1cm}\left[ \left( \begin{array}{c}
x^\ast \\ z^\ast 
\end{array} \right)\hspace{-0.1cm}+ \hspace{-0.1cm}D^\ast G(\ox,\oz,-\oy)(y^\ast)\;\Bigg|\;\begin{array}{c}
\left( \begin{array}{c}
x^\ast \\ z^\ast 
\end{array} \right)  \in D^\ast \Intfset{\Phi}( \ox,\oz,\oy )(y^\ast), \\
\oy \in  -G(\ox,\oz)\cap \Intfset{\Phi}(\ox,\oz)
\end{array} 	\right]
\end{align}
provided that for   all $\oy \in -G(\ox,\oz) \cap \Intfset{\Phi}(\ox,\oz)$ we have the qualification condition
\begin{align*}
N\big((\ox,\oz,\oy);\gph \Intfset{\Phi}\big)\cap \ker  D^\ast H	(\ox,\oz,\oy) =\{ 0\}. 
\end{align*}
Let us check that the latter conditions holds if the adjoint generalized equation \eqref{adj_gen_eq} admits only the trivial solution  $y^\ast=0$ for all $\oy \in -G(\ox,\oz)\cap \Intfset{\Phi} (\ox,\oz)$. Indeed, taking any $(x^\ast,z^\ast,-y^\ast ) \in N((\ox,\oz,\oy); \gph \Intfset{\Phi})$ and using the limiting coderivative definition lead us to $D^*H(\ox,\oz,\oy)(x^*,z^*,-z^*)=0$. This yields by \eqref{codH} and the coderivative Leibniz rule from Theorem~\ref{Theo_Regular_coder00}(iii) the fulfillment of the adjoint generalized equation \eqref{adj_gen_eq}, and thus $y^\ast=0$ by our assumption. It follows from Theorem~\ref{LipLikeEphi}(a) under the imposed  integrable quasi-Lipschitzian assumption on $\Phi$ that $(x^\ast, z^\ast)=0$. To complete the proof of this theorem, it remains to apply again the coderivative Leibniz rule  from Theorem~\ref{Theo_Regular_coder00}(iii) but now to the coderivative on the right-hand side of \eqref{Theo_Par005}.\end{proof}\vspace*{-0.1in}

The next theorem addresses a special setting of the stochastic variational system \eqref{def_mapS}, where the set-valued mapping $G$ does not depend on $x$, and where the integrand $\Phi_t(x,z)$ is single-valued and locally Lipschitzian around the reference point. Besides the coderivative evaluation for the solution map $S$ in this case, we derive now sufficient conditions expressed via the given system data for each of the well-posedness properties under consideration: the Lipschitz-like and metric regularity ones. \vspace*{-0.05in} 

\begin{theorem}\label{Corollary:Parametric:main}
Let $S:\X \tto \Z$ be given in  \eqref{def_mapS}, and let $(\ox,\oz) \in \gph S$. Suppose that $\Phi_t$ is single-valued and integrably locally Lipschitzian   around $(\ox,\oz)$, and that $G$ does not depend on $x$. Denoting $\oy:=\Intfset{\Phi}(\ox,\oz)$ and  picking any $x^\ast \in  D^\ast S(\ox,\oz) (z^\ast)$, we claim that there exists  $y^\ast \in \Y $  such that  
\begin{align}\label{EQ01_Cor_Par}
\begin{pmatrix}
x^\ast \\ -z^\ast 
\end{pmatrix} \hspace{-0.09cm} \in  \displaystyle \hspace{-0.09cm}\int_T\partial \langle y^\ast, \Phi_t \rangle (\ox,\oz) \dmu  \hspace{-0.09cm}+ \hspace{-0.09cm}\begin{pmatrix}
0\\  D^\ast G(\oz,-\oy)(y^\ast) 
\end{pmatrix}
\end{align}
provided the fulfillment of the qualification condition 
\begin{align}\label{QC:Corollary:Parametric:main}
\ker \left[ \proj_{\X}\left(\int_T\partial \langle \cdot, \Phi_t \rangle (\ox,\oz)\right)\right]=\{ 0\}.
\end{align}
Furthermore, we have the following well-posedness properties under the imposed additional assumptions, respectively:\vspace*{-0.05in}
\begin{enumerate}
\item[\bf(i)] The solution map \eqref{def_mapS} is Lipschitz-like around $(\ox,\oz)$ if the partial adjoint generalized equation defined by
\begin{equation}\label{page}
0\in\proj_{\Z}\left(\int_T\partial \langle y^*,\Phi_t\rangle(\ox,\oz)\right)+D^*G(\oz,-\oy)(y^*)
\end{equation}
admits only the trivial solution $y^*=0$.

\item[\bf(ii)] The solution map \eqref{def_mapS} is metrically regular around $(\ox,\oz)$ if the mapping $G$ therein is Lipschitz-like around $(\oz,-\oy)$. 
\end{enumerate}
\end{theorem}\vspace*{-0.05in}
\begin{proof} Observe that the adjoint generalized equation \eqref{adj_gen_eq} admits in our case only the trivial solution. Indeed, the inclusion in \eqref{adj_gen_eq} reduces now to
$$
\begin{pmatrix}
0  \\ 0
\end{pmatrix} \in  \displaystyle \hspace{-0.05cm}\int_T\partial \langle y^\ast, \Phi_t \rangle (\ox) \dmu  \hspace{-0.05cm}+ \hspace{-0.05cm}\begin{pmatrix} 0
\\ D^\ast G(\oz,-\oy)(y^\ast) 
\end{pmatrix}
$$
for some $y^\ast \in \Y$. Then it follows from \eqref{QC:Corollary:Parametric:main} that $y^\ast=0$, and thus we deduce from
Theorem~\ref{Theorem:Parametric:main} under the general assumptions made that \eqref{EQ01_Cor_Par} holds.   

Let us further verify the well-posedness properties under the corresponding additional assumptions in (i) and (ii). It is easy to see that the solution map \eqref{def_mapS} is closed-graph around $(\ox,\oz)$ under the general assumptions made. Starting with assertion (i), we apply the coderivative criterion \eqref{cod-cr} for the Lipschitz-like property to the mapping $S$ around $(\ox,\oz)$, which reads as $D^*S(\ox,\oz)(0)=\{0\}$. This means by the coderivative evaluation of Theorem~\ref{Theorem:Parametric:main} that $x^*=0$ the only vector satisfying the inclusion in \eqref{EQ01_Cor_Par} with $z^*=0$. By the structure of  \eqref{EQ01_Cor_Par} this is the case when the partial adjoint generalized equation \eqref{page} has only the trivial solution $y^*=0$. 

To prove assertion (ii), we pick any $z^\ast \in \Z$ such that $0\in D^\ast S(\ox,\oz)(z^\ast)$ and get by the coderivative evaluation 
in \eqref{EQ01_Cor_Par} that
\begin{equation}\label{svi1}
\begin{pmatrix}
0 \\-z^\ast 
\end{pmatrix} \in  \displaystyle \hspace{-0.05cm}\int_T\partial \langle y^\ast, \Phi_t \rangle (\ox) \dmu  \hspace{-0.05cm}+ \hspace{-0.05cm}\begin{pmatrix}
0\\  D^\ast G(\oz,-\oy)(y^\ast) 
\end{pmatrix},
\end{equation}
which implies by \eqref{QC:Corollary:Parametric:main} that $y^\ast=0$. Since $G$ is assumed to be Lipschitz-like around $(\oz,-\oy)$ we have $D^*G(\oz,-\oy)(0)=\{0\}$ by the coderivative criterion \eqref{cod-cr}, and therefore deduce from \eqref{svi1} that $z^*=0$. This verifies that $\ker D^*S(\ox,\oz)=\{0\}$ and thus confirms that $S$ is metrically regular around $(\ox,\oz)$ by the coderivative characterization of this property in \eqref{mr}.\end{proof}\vspace*{-0.05in}

Now we illustrate the application of Theorem~\ref{Corollary:Parametric:main} to the class of stochastic variational inequalities defined in Example~\ref{example_VI}.\vspace*{-0.05in}

\begin{example}\label{svi-ex} Consider the stochastic variational inequality in form \eqref{svi} with $(\ox,\oz) \in \gph S$, where $\Phi_\omega (x,z) = f(x) + H_\omega(z)$ with a continuously differentiable function $f$ around $\bar{x}$  and  a single-valued and integrable locally Lipschitz  function $H$ around $ \oz$. Suppose that the Jacobian matrix $\nabla f(\ox) $ is of full rank.  Then the qualification condition \eqref{QC:Corollary:Parametric:main} holds automatically. Therefore, Theorem~\ref{Corollary:Parametric:main} tells us that for every $x^\ast \in D^\ast S(\ox,\oz)(z^\ast)$ there exists a (unique) vector $y^\ast \in \Y$ such that $x^\ast=\nabla f(\ox)^*(y^\ast)$ and the inclusion 
\begin{align*}
-z^\ast &\in \displaystyle  \int_\Omega \partial \langle y^\ast , H_\omega \rangle  (\oz)\mathbb{P}(d\omega)+D^\ast G(\oz,-\oy)(y^\ast)
\end{align*}
is satisfied. To verify the Lipschitz-like property of \eqref{svi} from assertion (i) of Theorem~\ref{Corollary:Parametric:main}, observe that the mapping $G(\cdot)=N(z;K)$ does not depend on $x$ and thus the coderivative of $G$ in Theorem~\ref{Corollary:Parametric:main} for system \eqref{svi} reduces to $(0,D^*N(\cdot;K))^\top$ at the corresponding point, where $D^*N(\cdot;K)$ is actually the second-order subdifferential \eqref{2nd} of the indicator function $\dd(\cdot;K)$. Thus the mapping $S$ in this case is Lipschitz-like around $(\ox,\oz)$ if we have the implication 
\begin{equation*}
\Bigg[\displaystyle\proj_{\Z}\Bigg(\int_\Omega \partial \langle y^\ast, H_\omega \rangle  (\oz)\mathbb{P}(d\omega)\Bigg)+D^*N(\cdot;K)(\oz,-\oy)(y^*)\Bigg]\Longrightarrow y^*=0.
\end{equation*}
The condition in Theorem~\ref{Corollary:Parametric:main}(ii) for the metric regularity of $S$ reduces in the case of \eqref{svi} to the Lipschitz-like property of the normal cone mapping $z\mapsto N(z;K)$. However, it {\em does not hold} due to the intrinsic discontinuity of normal cone mappings; see the discussion right after Example~\ref{example_VI}.
\end{example}\vspace*{-0.03in}

Next we consider a more general class of {\em composite stochastic variational inequalities} defined by
\begin{align}\label{def_mapS03}
S(x):=\big\{ z \in \Z\big|\;0 \in \Intfset{\Phi}(x,z)+N\big(\psi(x,z);\Omega\big)  \big\},
\end{align}
where $\Phi: T\times \X \times \Z \to\Y$ is a vector-valued normal integrand, $\psi:\X \times \Z \to\Y$ is a vector-valued function, and $\Omega\subseteq Y$ is a closed set. As a consequence of Theorem~\ref{Corollary:Parametric:main} and coderivative calculus rules, we arrive at the following coderivative evaluation for system \eqref{def:set-valued:Exp0} in terms of its initial data.\vspace*{-0.05in}

\begin{corollary}\label{svi-comp}
Let  $S: \X \tto \Z$ be given in  \eqref{def_mapS03}, and let $(\ox,\oz)\in \gph S$ be such that $\Phi$ is single-valued and locally Lipschitz around $(\ox,\oz)$, and that $\psi$ is continuously differentiable around this point.  Denoting $\oy=\Intfset{\Phi}(\ox,\oz)$, we have that whenever $z^\ast \in \Z$ and $x^\ast \in 	D^\ast S(\ox,\oz) (z^\ast)$ there exists $y^\ast \in \Y$  for which the coderivative upper estimate
\begin{align*}
\begin{pmatrix}
x^\ast \\ -z^\ast 
\end{pmatrix} \in  \displaystyle \hspace{-0.05cm}\int_T\partial \langle y^\ast, \Phi_t \rangle (\ox,\oz) \dmu  + \nabla \psi(\ox,\oz)^\ast \circ D^\ast N(\cdot;\Omega)\big(\psi(\ox,\oz),\oy\big)(y^\ast)
\end{align*} holds, provided that the adjoint generalized equation 
\begin{equation*} 
0 \in  \int_T\partial \langle y^\ast, \Phi_t \rangle (\ox,\oz) \dmu  + \nabla \psi(\ox,\oz)^\ast \circ D^\ast N(\cdot;\Omega)\cdot\big(\psi(\ox,\oz),\oy\big)(y^\ast)
\end{equation*}
admits only the trivial solution  $y^\ast=0$. 
\end{corollary}\vspace*{-0.05in}
\begin{proof} This follows from Theorem~\ref{Corollary:Parametric:main} with $G(x,z):=N(\psi(x,z);\Omega)$ by applying the coderivative chain rule obtained in \cite[Theorem~3.11]{m18}.
\end{proof}\vspace*{-0.05in}

We conclude this section by the following consequence of Theorem~\ref{Corollary:Parametric:main} that addresses stochastic variational systems \eqref{def_mapS} with smooth mappings $\Phi_t$.\vspace*{-0.03in}

\begin{corollary}\label{Cor:Parametric:smooth} In addition to the general assumptions of Theorem~{\rm\ref{Corollary:Parametric:main}}, suppose that $\Phi_t$ is continuously differentiable around $(\ox,\oz)$ and that 
\begin{equation*}
\ker\left[\int_T \nabla_x \Phi_t(\ox,\oz)^\ast (\cdot) \dmu \right]=\{ 0\}.
\end{equation*}
Then for every $x^\ast \in D^\ast S(\ox,\oz) (z^\ast)$ there exists $y^\ast \in \Y$ such that 
\begin{align*}
x^\ast  &= \displaystyle \int_T \nabla_x \Phi_t(\ox,\oz)^\ast  (y^\ast) \dmu\;\mbox{ and }\\
-z^\ast &\in \displaystyle   \int_T \nabla_z \Phi_t(\ox,\oz)^\ast  (y^\ast) \dmu + D^\ast G(\oz,-\oy)(y^\ast),
\end{align*}
where $\oy: =\Intfset{\Phi}(\ox,\oz)$. Furthermore, the solution map $S$ from \eqref{def_mapS} is Lipschitz-like around $(\ox,\oz)$ if we have the implication
\begin{equation*}
\Bigg[0\in \displaystyle\int_T \nabla_z \Phi_t(\ox,\oz)^\ast  (y^\ast)\dmu + D^\ast G(\oz,-\oy)(y^\ast)\Big]]\Longrightarrow y^*=0. 
\end{equation*}
The metric regularity of the solution map $S$ around $(\ox,\oz)$ is ensured by the Lipschitz-like property of the mapping $G$ around $(\oz,-\oy)$.
\end{corollary}\vspace*{-0.05in}
\begin{proof} The claimed results follow from Theorem~\ref{Corollary:Parametric:main} due to the representation
$$
\int_T\partial \langle y^\ast, \Phi_t \rangle (\ox) \dmu  = \begin{pmatrix}
\displaystyle \int_T \nabla_x \Phi_t(\ox,\oz)^\ast  (y^\ast) \dmu\\ \vspace{-0.2cm}\\
\displaystyle \int_T \nabla_z \Phi_t(\ox,\oz)^\ast  (y^\ast) \dmu
\end{pmatrix}
$$
when $\Phi_t$ is continuously differentiable around $(\ox,\oz)$.
\end{proof}\vspace*{-0.35in}
	
\section{Sensitivity analysis under integrable Lipschitz-like property}\label{Sen_constraint_integrable_LipsLike}\vspace*{-0.1in}\setcounter{equation}{0}

In this section we examine the possibility of replacing the coderivative-based integrable quasi-Lipschitzian assumption \eqref{Int_Lips_like_inq} on the set-valued normal integrand $\Phi$ in the sensitivity analysis of stochastic constrained and variational systems in Section~\ref{Sensitivity_analysis_parametric_constraint} and Section~\ref{Sensitivity_variational}, respectively, by the explicit {\em integrable Lipschitz-like} property \eqref{eq_definition_Int_loc:Lipschitz-like} of this mapping. The reader is referred to \cite[Section~4]{mp3} for a detailed study of these properties and relationships between them.

For brevity, we present here only results concerning upper estimates of the limiting coderivative \eqref{lcod} for the general stochastic constraint mapping $F$ defined in \eqref{def_mapF} and the stochastic variational system $S$ defined in \eqref{def_mapS}. The corresponding conditions ensuring the Lipschitz-like and metric regularity properties of these systems can be derived, similarly to Sections~\ref{Sensitivity_analysis_parametric_constraint} and \ref{Sensitivity_variational}, from these evaluations due to the coderivative criteria. 
 
To proceed, we replace the assumptions in \eqref{convex_cond_set} with localized conditions at the point of interest. Given a random multifunction $\Phi: T\times \X \times\Z \tto \Y$ and a point  $(\ox,\oz)\in \dom  \Intf{\Phi}$, assume that there exits $\rho >0$ such that
\begin{align}\label{convex_cond_set2} 
\begin{aligned}
&\Phi_t(x,z) \text{ is convex for all } (x,z)\in \mathbb{B}_\rho(\ox,\oz) \text{ and all } t\in T,	\\
&\Phi_t(\ox,\oz)  \text{ is single-valued for all } t\in T.
\end{aligned}
\end{align}

The first result establishes a coderivative upper bound for the {\em stochastic constraint system} $F$ from \eqref{def_mapF}. Recall that when	$\Phi_t(\ox,\oz)=\{\ooy(t)\}$, we identify $\Phi_t(\ox,\oz)$ with $\ooy(t)$ and omit $\ooy(t)$ in the coderivative notation.\vspace*{-0.03in}

\begin{theorem}\label{COD:Lipschitz-like:F}
Let $F: \X\tto \Z$ be taken from \eqref{def_mapF} with a set-valued normal integrand  $\Phi: T\times \X\times \Z \tto \Y$ and closed sets $\mathcal{K} \subseteq\Z$ and $\mathcal{O}\subseteq \X\times \Y$. Pick $(\ox,\oz) \in \gph F$ and suppose that $\gph F$ is normally regular at $(\ox,\oz)$ and that $\Phi$ satisfies \eqref{convex_cond_set2}. Having $\oy \in \Intf{\Phi}(\ox,\oz) \cap \mathcal{K}$ and $\Phi_t(\ox,\oz)=\{ \ooy(t)\}$, assume that $\Phi$ is  integrably Lipschitz-like around $(\ox,\oz,\ooy)$. Then for all $z^\ast\in \Z$ and $x^\ast \in D^\ast F(\ox,\oz)(z^\ast)$ there exists $y^\ast \in N(\oy;\mathcal{K})$ such that
\begin{align}\label{incl:00}
\begin{pmatrix}
x^\ast\\ - z^\ast
\end{pmatrix}
\in\int_T D^\ast \Phi_t(x,z) (y^\ast) \dmu + 	N\big((\ox,\oz);\mathcal{O}\big)
\end{align}
provided the fulfillment of the qualification conditions
\begin{align}\label{Theo_Sens_general_Lipschitz-like}
\Bigg[0\in\int_T D^\ast \Phi_t(x,z) (y^\ast) \dmu,\;\;y^\ast \in 	N(\oy;\mathcal{K})\Bigg]\Longrightarrow y^\ast=0,
\end{align} 
\begin{align}
\bigcup_{y^*\in N(\oy;\mathcal{K})}\Bigg[\int_T D^\ast \Phi_t(\ox,\oz) (y^\ast) \dmu\Bigg]\bigcap\big(- N((\ox,\oz);\mathcal{O}\big)\Bigg]=\{0\}. \label{Theo_Sens_generalQC012}
\end{align} 
\end{theorem}
\begin{proof} Proceeding similarly to the proof of the truncation result from \cite[Theorem~3E.3]{dr}, we find positive constants  $\gamma,\eta$, and $\ell$ such that the truncated integrand $\Hat{\Phi}_t(x,z):= \Phi_t(x,z) \cap \mathbb{B}_{\gamma}(\ox,\oz)$ satisfies the integrable Lipschitzian property \eqref{eq_definition_Int_loc} around $(\ox,\oz)$ with constants $\eta$ and $\ell$. Define now the mapping 
\begin{equation*}
\Hat{F}(x):=\big\{z \in \Z\;\big|\;\Intfset{\Hat{\Phi}}(x,z) \cap \mathcal{K} \neq \emp, \; (x,z)\in \mathcal{O}\big\}.
\end{equation*}
Since $(\ox,\oz) \in \gph \Hat{F} \subseteq  \gph  {F}$ and $\gph F$ is normally regular at $(\ox,\oz)$, we get
\begin{equation*}
D^\ast F(\ox,\oz)(z^\ast)=\Hat D^\ast F(\ox,\oz)(z^\ast) \subseteq\Hat D^\ast \Hat{F}(\ox,\oz) (z^\ast)\subseteq D^\ast \Hat{F}(\ox,\oz) (z^\ast)
\end{equation*}
for all $z^*\in\Z$. Let us further check that the mapping $\Hat{\Phi}$ satisfies all the assumptions of Theorem~\ref{Theo_Sens_general}. Indeed, it is easy to see that \eqref{convex_cond_set02} holds for $\Hat \Phi$.  The  inner
semicontinuity of $\mathcal{S}_{\Hat{\Phi}}$ follows from \cite[Corollary~5.6]{mp3}. The quasi-Lipschitzian property of $\Hat\Phi$ is implied by its integrable Lipschitzian property by \cite[Theorem~4.4]{mp3}. Finally, the qualification conditions in  \eqref{Theo_Sens_generalQC02} follows from the assumptions \eqref{Theo_Sens_general_Lipschitz-like} and \eqref{Theo_Sens_generalQC012}, respectively, by using the fact that 
\begin{align}\label{eq00}
D^\ast\Phi_t(\ox,\oz)(y^\ast )= D^\ast\Hat{\Phi}_t(\ox,\oz)(y^\ast )\;\text{ for all }\;t\in T.
\end{align}
Therefore, applying Theorem~\ref{Theo_Sens_general} to the mapping $\Hat{F}$ with taking \eqref{eq00} into account, we arrive at \eqref{incl:00} and hence complete the proof of the theorem.
\end{proof}\vspace*{-0.07in}

The second theorem of this section addresses the {\em stochastic variational system} \eqref{def_mapS} and provides an upper estimate of its coderivative under the integrable Lipschitz-like property of the normal integrand $\Phi$ therein.\vspace*{-0.05in}

\begin{theorem}\label{svi-ll} Let  $S: \X \tto \Z$ be defined in \eqref{def_mapS} via  a set-valued normal integrand $\Phi: T\times \X \times\Z \tto \Y$ and a closed-graph multifunction $G: \X\times \Z\tto \Y$. Pick $(\ox,\oz) \in \gph S$ and assume that the set $\gph S$ is normally regular at $(\ox,\oz)$ and that the mapping $\Phi$ satisfies the conditions in \eqref{convex_cond_set2} and the integrable Lipschitz-like property \eqref{eq_definition_Int_loc:Lipschitz-like} around $(\ox,\oz,\ooy)$ with $\ooy(t)= \Phi_t(\ox,\oz)$. Fix $\oy \in  -G(\ox,\oz)\cap \Intfset{\Phi}(\ox,\oz)$ and then claim that for every $z^\ast \in \Z$ and $x^\ast \in  D^\ast S(\ox,\oz) (z^\ast)$ there exists  $y^\ast \in \Y$  such that
\begin{align*}
\left( \begin{array}{c}
x^\ast \\ -z^\ast 
\end{array} \right)\in\int_T D^\ast \Phi_t(\ox,\oz) (y^\ast) \dmu +  D^\ast G(\ox,\oz,-\oy)(y^\ast)
\end{align*}
provided that the adjoint generalized equation
\begin{equation*} 
0 \in  \int_T D^\ast \Phi_t(\ox,\oz) (y^\ast) \dmu + D^\ast G(\ox,\oz,-\oy)(y^\ast)
\end{equation*}
admits only the trivial solution  $y^\ast=0$.
\end{theorem}\vspace*{-0.05in}
\begin{proof}
Following the proof of Theorem~\ref{COD:Lipschitz-like:F}, we find a constant $\gamma>0$ such that the mapping $\Hat{\Phi}_t(x,z):= \Phi_t(x,z) \cap \mathbb{B}_{\gamma}(\ox,\oz)$ is integrably Lipschitzian around $(\ox,\oz)$. Define further the truncated stochastic variational system
\begin{equation}\label{S-trunc}
\Hat{S}(x):=\big\{ z \in \Z\;\big|\;0\in\Intfset{\hat{\Phi}}(x,z)+ G(x,z)\big\}.
\end{equation}
Then proceeding as in the proof of Theorem~\ref{COD:Lipschitz-like:F}, we  reduce the situation to the usage of Theorem~\ref{Theorem:Parametric:main} in the case of the truncated system \eqref{S-trunc}. 
\end{proof}\vspace*{-0.35in}
	
\section{Applications to problems of stochastic optimization}\label{Appli_Stochastic}\setcounter{equation}{0}\vspace*{-0.05in}

The concluding section of this paper addresses applications of the coderivative-based sensitivity analysis of general stochastic variational systems developed above to some important aspects of stochastic optimization. We split this section into two subsections. The first one deals with {\em stationary solution maps} for general problems of {\em constrained stochastic programming}, while in  the second subsection we derive new coderivative-type {\em necessary optimality conditions} for problems of {\em stochastic programming with equilibrium constraints}.\vspace*{-0.25in}

\subsection{\bf Coderivatives of solution maps in stochastic programming}\vspace*{-0.05in}

Given a normal integrand $f: T\times \X\times \Z\to \Rex$ and an extended-real-valued l.s.c.\ function $\psi :\X\times \Z \to\Rex$, consider the following class of {\em parametric stochastic programs} defined in the unconstrained format
\begin{align}\label{Problem01}
\min\limits_{ z\in \Z} \Intf{f}(x,z) + \psi(x,z)  
\end{align}
with the parameter $x\in\X$. Since $\psi$ is an extended-real-valued function, the model in \eqref{Problem01} implicitly encompasses problems of {\em constrained} stochastic programming defined in the form
\begin{align}\label{csp}
\min \Intf{f}(x,z) & \hspace{0.2cm}\text{ subject to }\;g (x,z) \in \mathcal{C},
\end{align}
where $\mathcal{C}\subseteq\Y$ is a closed set, and where $g\colon\X\times
\Z\to\Y$ is a continuous function. Indeed, the latter problem reduces to \eqref{Problem01} with $\psi:=\delta_{\mathcal{C}}\circ g$. In particular, programs of type \eqref{csp} include problems of {\em stochastic conic programming}, where $C$ is a closed convex cone; see, e.g., \cite{bs,sdr} and the references therein.

Applying the subdifferential Fermat and sum rules from \cite{m18} to local minimizers of \eqref{Problem01} under appropriate qualification conditions and then using the basic subdifferential Leibniz rule from \cite{mp3} for the expected-integral functional $\Intf{f}$, we get a collection of stationary points of the stochastic programming problem \eqref{Problem01} formalized as the {\em parametric stationary point map}
\begin{align}\label{def_mapS02}
S(x):=\Bigg\{  z \in \Z\;\Bigg|\; 0  \in   \int_T \partial_z f_t(x ,z) \dmu   +\partial_z \psi(x,z )\Bigg\}.
\end{align}
Note that for {\em convex} stochastic programs \eqref{Problem01}, the stationary point map \eqref{def_mapS02} equivalently describes the  parametric set of {\em optimal solutions} to \eqref{Problem01}, which is not the case for general stochastic problems of our interest here. Observe furthermore that \cite[Proposition~7.1]{mp3} (see also \cite{chp1,chp2,chp3} for related results) ensures under some additional assumptions that there exists $\eta>0$ such that  
\begin{align*}
\partial_z \Intf{f}(x,z)	=  \int_T \partial_z f_t(x ,z) \dmu, \text{  for   all }(x,z)\in \mathbb{B}_\eta(\ox,\oz).
\end{align*}
This tells us that \eqref{def_mapS02} describes the set of {\em M-stationary points} for \eqref{Problem01} that are important for various applications; see, e.g., \cite{hr} with other references.\vspace*{0.03in}

To proceed with the study of the stationary point map \eqref{def_mapS02} by reducing it to the stochastic variational system \eqref{def_mapS} investigated in Section~\ref{Sensitivity_variational}, we first formulate the following assumptions on the initial data of \eqref{Problem01} ensuring the fulfillment of those in \eqref{convex_cond_set02}:
there exist  $\eta>0$ and $\kappa \in \Leb^1(T,\R)$ such that 
\begin{equation}\label{cond_reg_sub0} 
\begin{aligned}
\partial_z f_t (x,z)  &\text{ is convex for all } (x,z)\in \mathbb{B}_\eta(\ox,\oz) \text{ and }  t\in T_{na},
\\
\partial_z f_t (x,z) &\subseteq \kappa(t) \mathbb{B} \text{  for all } (x,z) \in \mathbb{B}_\rho(\ox,\oz) \text{ and all } t\in T,
\end{aligned}
\end{equation}
where $T_{na }$ signifies the nonatomic part of $\mu$. Define the
set-valued mapping
\begin{align*}
\mathcal{S}_{f}(x,z,y):=\Bigg\{ \y \in \Leb^1(T,\Y)\;\Bigg|\;\int_T \y(t)\dmu =y \text{ and } \y(t) \in \partial_z f_t(x,z) \text{ a.e.}\Bigg\}.
\end{align*}

Our goal here is to present an efficient evaluation of the limiting coderivative of the stationary point map \eqref{def_mapS02} in terms of its initial data. The obtained results allow us to make a conclusion on the Lipschitz-like property of \eqref{def_mapS02}, which we skip for brevity. Recall that the {\em subdifferential structure} of the mapping $G(x,z):=\partial_z\psi(x,z)$ in \eqref{def_mapS02} prevents the fulfillment of the metric regularity property of this variational systems  {\em unless} $\psi$ is sufficiently smooth; see the discussions in Section~\ref{Sensitivity_variational} and Corollary~\ref{svi-mr} given below.\vspace*{0.03in}

According to Theorem~\ref{Theorem:Parametric:main}, we are going to apply the integrable quasi-Lipschitzian property \eqref{Int_Lips_like_inq} of the partial subdifferential mapping $(x,z)\mapsto\partial_z f_t(x,z)$, while indicating the possibility of using also the integrable Lipschitz-like property \eqref{eq_definition_Int_loc:Lipschitz-like} as in Theorem~\ref{svi-ll}, which is not done in this paper. The version of \eqref{Int_Lips_like_inq} corresponding to the structure of \eqref{def_mapS02} is formulated via the {\em partial second-order subdifferential} \eqref{2par} of $f$ as follows: there exist $\ell \in \Leb^1(T,\X)$ and  $\eta >0$  such that we have the second-order subdifferential estimate
\begin{align}\label{Extended:second_order_ Int_Lips_like_inq}
\sup\big\{ \| x^\ast\|\;\big|\;x^\ast \in{ \partial}_z^2f_t\big(\x(t),\y(t)\big)\big(\y^\ast(t)\big)\big\} \leq \ell(t)\|\y^\ast(t)\| \text{ a.e. }
\end{align}
for all $\x \in \mathbb{B}_\eta(\ox)$, $\y \in \mathbb{B}_\eta(\ooy)\cap \partial f(\x)$, and $\y^\ast\in \Leb^\infty(T,\Y)$, where 
\begin{align*}
\mathbb{B}_\eta(\ooy)\cap \partial f(\x)&:=\big\{ \y\in \Leb^1(T,\Y)\;\big|\;\y \in \mathbb{B}_\eta(\ooy) \text{ and } \y(t) \in \partial f_t\big(\x(t)\big)\text{ a.e.}\big\}.
\end{align*}\vspace*{-0.25in}

\begin{theorem}\label{Cor:Parametric:mapS02}
Let  $S: \X \tto \Z$ be given in \eqref{def_mapS02} with $(\ox,\oz) \in \gph S$, where $f$ satisfies \eqref{cond_reg_sub0}, and  where the mapping $(x,z) \to  - \partial_z \psi (x,z ) \cap \Intfset{ \partial_z f} (x,z)   $ is locally bounded around $(\ox,\oz)$. Assume furthermore that for all $ \y \in \mathcal{S}_{ \Phi} (\ox,\oz,\oy)$   with $\oy \in   -   \partial_z\psi(\ox,\oz ) \cap \Intfset{ \partial_z f} (\ox,\oz) $ we have the following conditions:\vspace*{-0.05in}
\begin{enumerate}[label=\alph*)]
\item[\bf(a)]\label{Cor:Parametric:mapS02:a} The mapping $\mathcal{S}_f$ is inner semicompact at $(\ox,\oz,\oy)$.
\item[\bf(b)]\label{Cor:Parametric:mapS02:b} The second-order subdifferential estimate \eqref{Extended:second_order_ Int_Lips_like_inq} is satisfied.
\end{enumerate}\vspace*{-0.05in}
Then whenever $z^\ast \in \Z$ and $x^\ast \in  D^\ast S(\ox,\oz) (z^\ast)$, there exist $y^\ast \in \Y$ and $\oy\in -\partial_z\psi(\ox,\oz ) \cap \Intfset{ \partial_z f} (\ox,\oz) $ such that
\begin{align*}
\left(\begin{array}{c}
x^\ast \\ -z^\ast 
\end{array} \right)\in \bigcup\limits_{ \y \in \mathcal{S}_{ f} (\ox,\oz,\oy)  } \int_T{\partial}_z^2 f_t     (\ox,\oz,\y(t)) (y^\ast) \dmu + {\partial}_z^2 (g \circ \psi  ) (\ox,\oz,-\oy)(y^\ast)
\end{align*}
provided that for all $\oy \in -   \partial_z\psi(\ox,\oz ) \cap \Intfset{ \partial_z f} (\ox,\oz)$  the second-order adjoint generalized equation defined by 
\begin{equation*} 
0 \in \bigcup\limits_{ \y \in \mathcal{S}_{ f} (\ox,\oz,\oy) } \int_T {\partial}_z^2 f_t(\ox,\oz,\y(t)) (y^\ast) \dmu + {\partial}_z^2  \psi  (\ox,\oz,-\oy)(y^\ast)
\end{equation*}
admits only the trivial	solution  $y^\ast=0$.
\end{theorem}\vspace*{-0.05in}
\begin{proof} It suffices to check that all the assumptions of Theorem~\ref{Theorem:Parametric:main} hold for the mapping $\Phi$ and $G$ therein defined by
\begin{align*}
\Phi_t(x,z):&=\left\{ \begin{array}{cc}
\partial_z f_t(x,z)&  \text{ if } (x,z) \in \mathbb{B}_\eta(\ox,\oz)\\
\emp&  \text{ if } (x,z) \notin \mathbb{B}_\eta(\ox,\oz)
\end{array}\right.,\\
G(x,z):&=\left\{ \begin{array}{cc}
\partial_z \psi(x,z)&  \text{ if }(x,z) \in \mathbb{B}_\eta(\ox,\oz)\\
\emptyset &  \text{ if } (x,z) \notin \mathbb{B}_\eta(\ox,\oz)
\end{array}\right..
\end{align*}
This readily follows from the definitions and constructions above.
\end{proof}\vspace*{-0.05in}

Finally, we present a simple consequence of Theorem~\ref{Cor:Parametric:mapS02} ensuring the metric regularity of the stationary point map \eqref{def_mapS02} in the case of {\em ${\cal C}^{1,1}$ functions} $\psi$ therein. Recall that this class consists of continuously differentiable functions with locally Lipschitzian gradients around the reference point.\vspace*{-0.05in} 

\begin{corollary}\label{svi-mr} In addition to the assumption of Theorem~{\rm\ref{Cor:Parametric:mapS02}}, suppose that $\psi$ does not depends on $x$ being of class ${\cal C}^{1,1}$ around $\oz$, and that
\begin{align*}
\ker \left[ \proj_{\X}\left(\bigcup\limits_{ \y \in \mathcal{S}_{ f} (\ox,\oz,\oy)  } \int_T{\partial}_z^2 f_t\big(\ox,\oz,\y(t)\big)(\cdot)\dmu  \right)
\right]=\{0\}, 
\end{align*}
where  $\oy: =-\nabla  \psi   (\oz ) $.	Then $S$ in \eqref{def_mapS02} is metrically regular around $(\ox,\oz)$.
\end{corollary}\vspace*{-0.05in}
\begin{proof}
Using Theorem~\ref{Cor:Parametric:mapS02} and the proof of Theorem~\ref{Corollary:Parametric:main}, we readily check that $\ker D^\ast S(\ox,\oz)=\{ 0\}$, which yields the metric regularity of $S$ around $(\ox,\oz)$ by the coderivative characterization in \eqref{mr}. 
\end{proof}\vspace*{-0.35in}
 
\subsection{\bf Stochastic programming  with equilibrium constraints}\vspace*{-0.05in}

Various versions of {\em stochastic mathematical programs with equilibrium constraints} (stochastic MPECs) have been formulated and investigated in the literature with numerous applications to practical models; see, e.g.,  \cite{hr,pw,sdr} and the references therein. the underlying feature of stochastic MPECs in the presence among constraints stochastic generalized equations of type \eqref{def_mapS} and their specifications, which can be interpreted as certain equilibrium conditions. In this subsection we consider a general class of such problems given by
\begin{equation}\label{Proble:MPECs}
\begin{aligned}
\min \varphi(x,z)
\text{ subject to }  z \in S(x), \; x \in \mathcal{C},
\end{aligned}
\end{equation}
where $S(x)$ is a stochastic parametric variational system taken from \eqref{def_mapS} in the setting described therein, where 
$\varphi: T\times \X\times \Z\to \Rex$ is an l.s.c.\ cost function, and where ${\cal C}$ is a closed set. Our goal here is to derive {\em necessary optimality conditions} for such stochastic MPECs by using the basic machinery of variational analysis and the coderivative evaluation for the stochastic system \eqref{def_mapS} obtained in Section~\ref{Sensitivity_variational}. 

To formulate desired necessary optimality conditions for the stochastic MPEC \eqref{Proble:MPECs} in full generality, we need to recall yet another subdifferential construction for extended-real-valued function. Given 
$f\colon\X\to\oR$ finite at $x$, its {\em singular subdifferential} at this point is defined geometrically via the limiting normal cone \eqref{lnc} to the epigraph of $f$ by
\begin{equation}\label{ss}
\partial^\infty f(x):=\big\{x^*\in\X\;\big|\;(x^*,0)\in N\big((x,f(x);\epi f\big)\big\}.
\end{equation}
We refer the reader to the books \cite{m06,m18,rw} and the bibliographies therein for various properties, analytic representations, and calculus rules for the singular subdifferential \eqref{ss}. Note that if $f$ is l.s.c.\ around $x$, then $\partial^\infty f(x)=\{0\}$ {\em if and only if} $f$ is locally Lipschitzian around this point.\vspace*{0.03in} 

Now we are ready to establish the aforementioned necessary optimality conditions for the stochastic MPEC \eqref{Proble:MPECs}.\vspace*{-0.05in}

\begin{theorem}\label{nec:Proble:MPECs}
Let $(\ox,\oz) \in \X\times \Y$ be a local optimal solution to \eqref{Proble:MPECs}. In addition to the assumption of 
Theorem~{\rm\ref{Theorem:Parametric:main}} with $D^*S(\ox,\oz)$ evaluated therein, suppose that $(x^*_\infty,z^*_\infty)=(0,0)$ is the only pair satisfying the inclusions
\begin{align}\label{QC:nec:Proble:MPECs}
(x_\infty^\ast,z^\ast_\infty )\in \partial^\infty \varphi(\ox,\oz)\;\mbox{ and }\;0 \in x^\ast_\infty + D^\ast S(\ox,\oz) (z^\ast_\infty ) + N( \ox;\mathcal{C}),
\end{align} which are automatically fulfilled if $\varphi$ is locally Lipschitzian around $(\ox,\oz)$. Then there exist elements $(x^\ast, z^\ast) \in \partial \varphi(\ox,\oz)$, $y^\ast \in \Y$, $\oy \in -G(\ox,\oz)\cap \Intfset{\Phi}(\ox,\oz)$,  and $\y \in \mathcal{S}_{ \Phi} (\ox,\oz,\oy)$ such that we have
\begin{align*}
0&\in \left( \begin{array}{c}
x^\ast \\ z^\ast 
\end{array} \right)+ D^\ast G(\ox,\oz,-\oy)(y^\ast) + \int_T D^\ast \Phi_t\big(\ox,\oz,\y(t)\big)(y^\ast) \dmu +\begin{pmatrix}
N(\ox;\mathcal{C})\\ 0
\end{pmatrix}.
\end{align*}
\end{theorem}\vspace*{-0.05in}
\begin{proof} It follows from \cite[Theorem~5.34]{m06} that there exist dual vectors $(x^\ast, z^\ast) \in \partial \varphi(\ox,\oz)$ and $w^\ast \in D^\ast S(\ox,\oz)(z^\ast) $ for which
\begin{align}\label{nec:Proble:MPECs:Eq01}
0 \in x^\ast +w^\ast + N(\ox;\mathcal{C})
\end{align}
provided that the qualification condition \eqref{QC:nec:Proble:MPECs} is satisfied. Now using the coderivative evaluation from Theorem~\ref{Theorem:Parametric:main}, we find $y^\ast \in \Y$, $\oy \in -G(\ox,\oz)\cap\Intfset{\Phi}(\ox,\oz)$, and $\y \in \mathcal{S}_{ \Phi} (\ox,\oz,\oy)$ such that
\begin{align}\label{nec:Proble:MPECs:Eq02}
\begin{pmatrix}
w^\ast \\ -z^\ast
\end{pmatrix} \in D^\ast G(\ox,\oz,-\oy)(y^\ast) + \int_T D^\ast \Phi_t\big(\ox,\oz,\y(t)\big)(y^\ast) \dmu.
\end{align}
Combining finally the relationships in \eqref{nec:Proble:MPECs:Eq01} and \eqref{nec:Proble:MPECs:Eq02} verifies the claimed necessary optimality conditions for the stochastic MPEC \eqref{Proble:MPECs}.
\end{proof}\vspace*{-0.1in}

The last result is a direct consequence of Theorem~\ref{nec:Proble:MPECs}.\vspace*{-0.1in}

\begin{corollary} Considering a local optimal solution $(\ox,\oz)$ to \eqref{Proble:MPECs} with $G$ in \eqref{def_mapS} independent on $x$, assume that $\ph$ is locally Lipschitzian around $(\ox,\oz)$ while $\Phi$ is integrably locally Lipschitzian around this point. Denoting $\oy:=\Intfset{\Phi}(\ox,\oz)$, we claim that there exists $y^\ast \in \Y$ such that
\begin{align*}
0&\in \partial \varphi (\ox,\oz)+ \int_T   \partial \langle y^\ast,  \Phi_t \rangle (\ox,\oz) \dmu   +\begin{pmatrix}
N(\ox;\mathcal{C})\\  D^\ast G(\oz,-\oy)(y^\ast)
\end{pmatrix}.
\end{align*}
\end{corollary}\vspace*{-0.1in}
\begin{proof} Since $\ph$ is locally Lipschitzian around $(\ox,\oz)$, the qualification condition \eqref{QC:nec:Proble:MPECs} holds automatically. Furthermore, the coderivative representation for the integrably Lipschitzian mapping $\Phi$ in \eqref{nec:Proble:MPECs:Eq02} follows from the scalarization formula \eqref{scal}. Combining these facts, we justify our claim.
\end{proof} 
\end{document}